\documentclass[12pt]{amsart}
\usepackage{amssymb}
\usepackage{mathtools}
\usepackage{xcolor}

\newcommand{\REM}[1]{\relax}

\newcommand{\Card}{\mathsf{Card}}

\begin{document}
\numberwithin{equation}{section}

\newcommand{\anglim}{\angle\lim}

\newcommand{\de}{\partial}
\newcommand{\Hol}{{\sf Hol}}
\newcommand{\Aut}{{\sf Aut}(\mathbb D)}
\newcommand{\loc}{L^1_{\rm{loc}}}
\newcommand{\Moeb}{\mathsf{M\ddot ob}}

\newcommand{\Mobius}{M\"obius }

\def\Re{{\sf Re}\,}
\def\Im{{\sf Im}\,}

\newcommand{\UH}{\mathbb{H}}
\newcommand{\UHi}{\mathbb{H}_i}
\newcommand{\Real}{\mathbb{R}}
\newcommand{\Natural}{\mathbb{N}}
\newcommand{\Complex}{\mathbb{C}}
\newcommand{\ComplexE}{\overline{\mathbb{C}}}
\newcommand{\Int}{\mathbb{Z}}

\newcommand{\UD}{\mathbb{D}}
\newcommand{\UC}{\partial\UD}
\newcommand{\clS}{\mathcal{S}}
\newcommand{\gtz}{\ge0}
\newcommand{\gt}{\ge}
\newcommand{\lt}{\le}
\newcommand{\fami}[1]{(#1_{s,t})}
\newcommand{\famc}[1]{(#1_t)}
\newcommand{\ts}{t\gt s\gtz}
\newcommand{\classCC}{\tilde{\mathcal C}}
\newcommand{\classS}{\mathcal S}
\newcommand{\U}{{\mathfrak U}}
\newcommand{\Ut}[1]{{\mathfrak U\hskip.1em}_{#1}\hskip-.09em}
\newcommand{\UF}{{\mathfrak U\hskip.08em}}
\newcommand{\Un}{{\mathfrak U\hskip.07em}_0\hskip-.07em}

\newcommand{\Maponto}
{\xrightarrow{\scriptstyle \!\mathsf{onto}\,}}

\let\N=\Natural
\let\R=\Real
\let\D=\UD
\let\RS=\ComplexE
\let\C=\Complex
\newcommand{\uH}{\mathbb{H}_{\mathrm{i}}}
\newcommand{\rH}{\mathbb{H}_1}
\newcommand{\mscrm}{\mathfrak{m}}
\newcommand{\mscrc}{\mathfrak{c}}
\newcommand{\mathup}[1]{\mathrm{#1}}

\newcommand{\uHc}{\overline{\mathbb{H}}_{\mathrm{i}}}

\newcommand{\Step}[2]{\begin{itemize}\item[{\bf Step~#1.}]\textit{#2}\end{itemize}}
\newcommand{\step}[2]{\begin{itemize}\item[\textit{Step\,#1:}]\textit{#2}\end{itemize}}
\newcommand{\case}[2]{\begin{itemize}\item[\textit{Case~#1:}]\textit{#2}\end{itemize}}
\newcommand{\caseM}[1]{\par\vskip.7ex\noindent\textit{Case~#1}.~}
\newcommand{\caseU}[2]{\vskip1.4ex\begin{trivlist}%
\item[\textbf{CASE~#1}:]\textbf{{\mathversion{bold}#2}}\end{trivlist}\vskip-0.3ex}
\newcommand{\parM}[1]{\par\vskip.7ex\noindent\textit{#1}.~}
\newcommand{\proofbox}{\hfill$\Box$}
\newcommand{\trivlistcorr}{\leftmargin=2.3em\labelsep=0em\labelwidth=2.3em\rightmargin=.5em\itemindent=0em}
\newcommand{\trivlistcorrM}{\leftmargin=2.3em\labelsep=0.5em\labelwidth=1.8em\rightmargin=.5em\itemindent=0em}
\newcommand{\claim}[2]{\vskip1.4ex\begin{list}{}{\trivlistcorr}\item[#1\hfill] {\it #2}\end{list}\vskip.5ex}
\newcommand{\claimM}[2]{\vskip1.4ex\begin{list}{}{\trivlistcorrM}\item[\textbf{Claim #1}:~] {\it #2}\end{list}}

\newcommand{\STOP}{\par\hbox to\textwidth{\color{red}\leaders\hbox{\,STOP\,}\hfil}\par}

\newcommand{\mcite}[1]{\csname b@#1\endcsname}

\newcommand{\di}{\mathrm{d}}
\newcommand{\DI}{\,\di}
\newcommand{\ddt}[1]{\frac{\di #1}{\di t}}
\newcommand{\DDT}{\ddt{}\,}

\newcommand{\diam}{\mathop{\mathsf{diam}}\nolimits}

\theoremstyle{theorem}
\newtheorem {result} {Theorem}
\setcounter {result} {64}
 \renewcommand{\theresult}{\char\arabic{result}}



\newcommand{\Spec}{\Lambda^d}
\newcommand{\SpecR}{\Lambda^d_R}
\newcommand{\Prend}{\mathrm P}




\def\cn{{\C^n}}
\def\cnn{{\C^{n'}}}
\def\ocn{\2{\C^n}}
\def\ocnn{\2{\C^{n'}}}
\def\je{{\6J}}
\def\jep{{\6J}_{p,p'}}
\def\th{\tilde{h}}


\def\dist{{\rm dist}}
\def\const{{\rm const}}
\def\rk{{\rm rank\,}}
\def\id{{\sf id}}
\def\aut{{\sf aut}}
\def\Aut{{\sf Aut}}
\def\CR{{\rm CR}}
\def\GL{{\sf GL}}
\def\Re{{\sf Re}\,}
\def\Im{{\sf Im}\,}
\def\U{{\sf U}}

\def\Cap{\mathop{\mathrm{cap}}}

\def\la{\langle}
\def\ra{\rangle}

\newcommand{\sgn}{\mathop{\mathrm{sgn}}}

\emergencystretch15pt \frenchspacing

\newtheorem{theorem}{Theorem}
\newtheorem{lemma}{Lemma}[section]
\newtheorem{proposition}{Proposition}
\newtheorem{corollary}{Corollary}

\theoremstyle{definition}
\newtheorem{definition}[lemma]{Definition}
\newtheorem{example}[lemma]{Example}

\theoremstyle{remark}
\newtheorem{remark}[lemma]{Remark}
\numberwithin{equation}{section}

\newcommand{\res}{\mathop{\mathrm{\,r\,e\,s\,}}\limits}

\newcommand{\distCE}{\mathop{\mathsc{dist}_\ComplexE}}

\newenvironment{mylist}{\begin{list}{}%
{\labelwidth=2em\leftmargin=\labelwidth\itemsep=.4ex plus.1ex
minus.1ex\topsep=.7ex plus.3ex
minus.2ex}%
\let\itm=\item\def\item[##1]{\itm[{\rm ##1}]}}{\end{list}}

\title[Parametric representation of univalent functions with BRFPs]{Parametric representation of univalent functions with boundary regular fixed points}
\author[P. Gumenyuk]{Pavel Gumenyuk$^*$}
\address{Department of Mathematics and Natural Sciences, University of Stavanger, N-4036 Stavanger, Norway.}
\email{pavel.gumenyuk@uis.no}



\thanks{$^*$\,Partially supported by the FIRB grant Futuro in Ricerca  ``Geometria Differenziale Complessa e
Dinamica Olomorfa" n.\,RBFR08B2HY and by \textit{Ministerio de Econom\'\i a y Competitividad (Spain)} project \hbox{MTM2015-63699-P}.}

\begin{abstract}
A classical result in the theory of Loewner's parametric representation states that
the semigroup $\Un$ of all conformal self-maps $\varphi$ of the unit disk $\D$ normalized by $\varphi(0) = 0$ and
$\varphi'(0) > 0$ can be obtained as the reachable set of the Loewner\,--\,Kufarev control system
$$
\frac{\di w_t}{\di t}=G_t\circ w_t,\quad t\geqslant0,\qquad w_0=\id_\UD,
$$
where the control functions $t\mapsto G_t\in\Hol(\UD,\C)$ form a convex cone. We extend this result to semigroups
$\UF[F]$ formed by all conformal self-maps of $\D$ with the prescribed finite set~$F$ of boundary regular fixed points and to their counterparts $\Ut\tau[F]$ for the case of self-maps having the Denjoy\,--\,Wolff point at~$\tau\in\overline\UD\setminus F$.
\end{abstract}

\maketitle

\tableofcontents

\let\ge=\geqslant
\let\le=\leqslant

\section{Introduction}
 Injectivity of a map  is a non-linear property. This fact leads to considerable complication in the study of univalent (\textit{i.e.}, injective holomorphic) functions. One of the tools that helps to overcome this difficulty is \textsl{Loewner's parametric representation.}  The most known version of Loewner's method represents the class~$\clS$ of all univalent functions $f$ in the unit disk ${\UD:=\{z:|z|<1\}}$ normalized by~$f(0)=0$, $f'(0)=1$, as the image of a convex cone w.r.t. a map defined via solutions to a differential equation.
This representation of the class~$\clS$ as well as many other applications of Loewner Theory, including recently discovered stochastic counterpart of Loewner's differential equation~\cite{Schramm}, deal mainly with families of conformal mappings onto nested plain domains, known as Loewner chains.

The framework of this paper lies within a different approach, originally due to Loewner himself~\cite{Loewner}, which has been systematically developed by Goryainov~[\mcite{Goryainov}\,--\,\mcite{Goryainov-Kudryavtseva}]. It is based on the simple idea that injectivity is preserved under composition. In particular,
the class~$\mathfrak U$ of all univalent  self-maps of~$\UD$ is  a semigroup w.r.t. the operation $(\psi,\varphi)\mapsto\psi\circ\varphi$.  Moreover, endowed with the usual topology of locally uniform convergence, $\mathfrak U$ becomes a topological semigroup. Berkson and Porta~\cite{Berkson-Porta} proved that every one-parameter semigroup~$(\phi_t)_{t\ge0}$ in~$\mathfrak U$, \textit{i.e.} a continuous semigroup homomorphism $\big([0,+\infty),\cdot+\cdot\big)\ni t\mapsto \phi_t\in\mathfrak U$, is a semiflow defined by the initial value problem
\begin{equation}\label{EQ_autonomousFlow}
 \DDT\phi_t(z)=G(\phi_t(z)),\quad t\ge0;\qquad \phi_0(z)=z\in\UD,
\end{equation}
where $G$ is a holomorphic function called the \textsl{(infinitesimal) generator} of~$(\phi_t)$. The set $\mathrm{T}\mathfrak U$ of all generators, which we will call the \textsl{infinitesimal structure} of the semigroup~$\mathfrak U$, is a convex cone described by the Berkson\,--\,Porta formula~\cite{Berkson-Porta},
\begin{equation}\label{EQ_BP}
G(z)=(\tau-z)(1-\overline\tau z)p(z),\quad \text{for all }z\in\UD,
\end{equation}
where $\tau\in\overline{\UD}$ and $p\in\Hol(\UD,\Complex)$ with $\Re p\ge0$.

Up to a certain extent, the map
\begin{equation}\label{EQ_exp}
\mathrm T\mathfrak U\ni G\mapsto\phi_1^G\in\mathfrak U,
\end{equation}
 where $(\phi_t^G)$ is the one-parameter semigroup whose generator is~$G$,  can be thought as an analogue of the exponential map in the theory of Lie groups, with $G$ playing the role of a tangent vector at the identity.

 However, our setting is quite different from that of (finite- or infinite-dimensional) Lie groups. In particular, the image the map~\eqref{EQ_exp} does not cover even a neighbourhood of~$\id_\UD$ in~$\mathfrak U$,  see, \textit{e.g.}, \cite[Corollary~3.5 and text below]{Fractional}. This makes impossible to recover the semigroup~$\mathfrak U$ from its infinitesimal structure~$\mathrm{T}\mathfrak U$ by means of~\eqref{EQ_exp}.
The same takes place if we restrict ourselves to the subsemigroup~$\Un:=\{\varphi\in\mathfrak U:\varphi(0)=0,~\varphi'(0)>0\}$.

Loewner~\cite{Loewner} realized that it is still possible to recover~$\Un$ from~$\mathrm{T}\Un$ if one considers a non-autonomous analogue of~\eqref{EQ_autonomousFlow} by replacing $G$ with a family $0\le t\mapsto G(\cdot,t)\in\mathrm T\Un$. Here $\mathrm T\Un$ stands for the convex cone formed by all generators~$G$ satisfying $(\phi^G_t)\subset\Un$, namely
$$
\mathrm T\Un=\big\{~\UD\ni z\mapsto\,-zp(z):~p\in\Hol(\UD,\C),~\Re p\ge0,~\Im p(0)=0~\big\}
$$
thanks to~\eqref{EQ_BP}. This leads to the initial value problem
\begin{equation}\label{EQ_LK-ODE-intro}
\DDT\varphi_{s,t}(z)=-\varphi_{s,t}(z)\,p\big(\varphi_{s,t}(z),t\big),\quad t\ge0;\qquad \varphi_{s,s}(z)=z\in\UD,
\end{equation}
where $G(z,t):=-zp(z,t)$ belongs to~$\mathrm T\Un$ for a.e.~$t\ge0$ and satisfies the conditions in Definition~\ref{DF_HVF}.

Equation~\eqref{EQ_LK-ODE-intro} is the classical Loewner\,--\,Kufarev ODE, see, \textit{e.g.}, \cite{Aleksandrov} or~\cite[Chapter~6]{Pommerenke}. It is known that the union of all non-autonomous semiflows~$(\varphi_{s,t})$ of~\eqref{EQ_LK-ODE-intro} corresponding to various choices of~$p$, coincides with~$\Un$. This fact is the essence of Loewner's parametric representation of~$\Un$ and it was used by de Branges in his proof~\cite{deBranges} of Bieberbach's famous conjecture.
If now we renormalize elements in~$(\varphi_{s,t})$, \textit{i.e.} if we consider $\varphi_{0,t}/\varphi_{0,t}'(0)$, and take the limit as $t\to+\infty$, we obtain the parametric representation of the class~$\clS$, which is usually meant when one refers to Loewner's parametric representation of conformal mappings\footnote{Loewner~\cite{Loewner} himself obtained the parametric representation of a dense subclass of~$\clS$. Later it was extended to the whole class by Pommerenke~\cite{Pommerenke-65, Pommerenke} and independently by Gutljanski\u\i~\cite{Gutljanski}.}.

Loewner's idea can be applied to other semigroups. Goryainov~\cite{Goryainov, GoryainovBRFP} and Goryainov and Ba~\cite{Goryainov-Ba}, see also \cite{Aleks1983, AleksST, Bauer}, established parametric representations for several subsemigroups $\mathfrak S\subset\mathfrak U$ based on reconstruction of a semigroup from its infinitesimal structure by means of a suitable analogue of the Loewner\,--\,Kufarev equation. Moreover, Loewner himself applied this method to a certain semigroup of matrices \cite{LoewnerMatrices}, see also~\cite{Chon}, and to monotone matrix functions~\cite{LoewnerSeminar,LoewnerMonotone}.
It is also worth to mention that very similar constructions appear in the study of subsemigroups of Lie groups in connection with Control Theory, see, \textit{e.g.}, \cite{LieSemigroups} and references therein.

However, it does not seem to be known any general idea of how one can determine whether a given semigroup~$\mathfrak S$ can be reconstructed from its infinitesimal structure. For subsemigroups of $\Hol(\UD,\UD)$ a necessary condition is univalence: thanks to the uniqueness of solutions to the initial value problem for the general version of the Loewner\,--\,Kufarev ODE (see equation~\eqref{EQ_general-LK-INI} in Sect.\,\ref{SS_param-represent}), all the functions $\varphi_{s,t}$ are univalent in~$\UD$. However, we do not know whether  the univalence of all elements of a semigroup~$\mathfrak S\subset\Hol(\UD,\UD)$ is ``close'' to being a sufficient condition.

Therefore, it seems to be reasonable to consider more examples of semigroups in~$\mathfrak U$ and try to answer the question whether they can be reconstructed from their infinitesimal structures. In this paper we deal with two families of semigroups: $\UF[F]$ and $\Ut\tau[F]$, defined as follows. For every finite set $F:=\{\sigma_1,\sigma_2,\ldots,\sigma_n\}\subset \partial\UD$, the semigroup $\UF[F]$ is the set of all~$\varphi\in\mathfrak U$ that have boundary regular fixed points at each~$\sigma_j$, $j=1,2,\ldots,n$, see Definition~\ref{DF_contact-BRFP}. Given $\tau\in\overline\UD\setminus F$, we define  $\Ut\tau[F]$ to be the subset of~$\UF[F]$ consisting of~$\id_\UD$ and all~$\varphi\in\UF[F]$ with the Denjoy\,--\,Wolff point at~$\tau$, see Definition~\ref{DF_DW}. We choose these families of semigroups by two main reasons. First of all, there has been an increasing interest in the study of holomorphic (injective and non-injective) self-maps of~$\UD$ with given boundary regular fixed points, see, \textit{e.g.}, \cite{Unkelbach1938,Unkelbach1940,CowenPommerenke,Poggi,PommVas,CDP1, MilnVas,ElinShoikhetTarkhanov,Frolova,GoryainovBRFP}. Secondly, the infinitesimal structure of such semigroups is well-studied, see, \textit{e.g.}, \cite{GoryainovFr,CDP2,AlexanderClarkMeasures,Goryainov-Kudryavtseva,ElinShoikhetTarkhanov}.

Note that there is no restriction on the number and choice of fixed points $\sigma_j$. In fact, for \textsl{any} finite set $F\subset\partial\UD$ and any $\tau\in\overline\UD\setminus F$, the semigroups $\UF[F]$ and $\Ut\tau[F]$ are non-trivial, see Example~\ref{EX_non-trivial}.

Our main result is the following theorem, see Sect.\,2 for the definition of Loewner-type parametric representation.

\begin{theorem}\label{TH_main}
The following semigroups $\mathfrak S$  admit Loewner-type parametric representation:
\begin{itemize}
\item[] $\mathfrak S:=\UF[F]$ with $\Card(F)\le 3$,\vskip.4ex
\item[] $\mathfrak S:=\Ut\tau[F]$ with $\tau\in\UC$ and $\Card(F)\le2$, and \vskip.4ex
\item[] $\mathfrak S:=\Ut\tau[F]$ with $\tau\in\UD$ and any finite set~$F\subset\UC$.
\end{itemize}
\end{theorem}

The Loewner-type parametric representation for the case $\mathfrak S=\Ut\tau[F]$ with ${\tau\in\UD}$ and ${\Card(F)=1}$
goes back to Unkelbach~\cite{Unkelbach1940}, who suggested a kind of discrete version of the parametric representation for~$\Un[\{1\}]$, yet with an important hypothesis left unproved, and used it to obtain a sharp estimate of $\varphi'(1)$ in terms of $\varphi'(0)$.
The Loewner-type parametric representation for $\Un[\{1\}]$ in the form of a differential equation, analogous to the Loewner\,--\,Kufarev ODE, was rigorously proved much later by Goryainov~\cite{GoryainovBRFP}.

To demonstrate a potential usage of Theorem~\ref{TH_main}, we obtain the analogue of the Loewner\,--\,Kufarev ODE for~$\Ut{\,1}[\{-1\}]$.

\begin{corollary}\label{CR_specific}
Let $T>0$. The class of all univalent $\varphi\in\Hol(\UD,\UD)$ with the Denjoy\,--\,Wolff point~$\tau=1$ and a boundary regular fixed point~$\sigma=-1$ of dilation~$\varphi'(\sigma)=e^{T}$ coincides with the set of all functions representable in the form $\varphi(z)=w_z(T)$ for all $z\in\UD$, where $w_z(t)$ is the unique solution to the initial value problem
\begin{equation}\label{EQ_LK-specific}
\ddt{w_z}=\tfrac14(1-w_z)^2(1+w_z)q(w_z,t),\quad t\in[0,T],\qquad w_z(0)=z,
\end{equation}
with some function $q:\UD\times[0,T]\to\C$ satisfying the following conditions:
\begin{itemize}
\item[(i)] for every $z\in\UD$, $q(z,\cdot)$ is measurable on~$[0,T]$;
\item[(ii)] for a.e.~$t\in[0,T]$, $q(\cdot,t)$ has the following integral representation
\begin{equation*}
q(z,t)=\int\limits_{\UC\setminus\{1\}}\!\!\frac{1-\kappa}{1+\kappa z}\,\DI\nu_t(\kappa),
\end{equation*}
where $\nu_t$ is a probability measure on $\UC\setminus\{1\}$.
\end{itemize}
\end{corollary}
\begin{remark}
It follows from \cite[Theorem\,1.1]{BRFPLoewTheory} and the proof of Corollary~\ref{CR_specific} that for any $t\in[0,T]$ the angular derivatives of $z\mapsto w_z(t)$ at $\sigma=-1$ and at~$\tau=1$ equal $e^{t}$ and $\exp(-\int_0^t\nu_s(\{-1\})\DI s)$, respectively.
\end{remark}

The paper is organized as follows. In Sect.\,\ref{S_Prelim} we collect some material from the theory of holomorphic self-maps of~$\UD$ and modern Loewner Theory necessary for our purposes. The main problem we address in this paper is stated in  Sect.\,\ref{S_PrStatement}, where we also reformulate it as a problem of embedding in evolution families, see Theorem~\ref{TH_embedd}.

 In Sect.\,\ref{S_diff_theorems}, we present some results on evolution families in~$\UF[F]$ and $\Ut\tau[F]$, in particular Theorem~\ref{TH_differentiability}, which we believe are closely related to the topic although they are not used in the proof of the main theorem. Finally, in Sect.\,\ref{S_PROOF}, we prove Theorem~\ref{TH_main} and Corollary~\ref{CR_specific}, while a question concerning cases not covered by Theorem~\ref{TH_main} is raised in Sect.\,\ref{S_OpenProblem}.

\section{Preliminaries}\label{S_Prelim}
\subsection{Regular fixed and contact points}\label{SS_RBFP}
For a function $f:\UD\to\Complex$ by $\anglim_{z\to\sigma}f(z)$, where $\sigma\in\UC$, we will denote the \textsl{angular} (called also \textsl{non-tangential}) \textsl{limit} of $f$ at~$\sigma$. If this limit exists, we will denote it, as usual, by~$f(\sigma)$. If, in addition, $\anglim_{z\to\sigma}\frac{f(z)-f(\sigma)}{z-\sigma}$ exists, finite or infinite, it is called the \textsl{angular derivative} of $f$ at~$\sigma$ and denoted by~$f'(\sigma)$. Note that if $f'(\sigma)\in\Complex$, then the limit $\anglim_{z\to\sigma}f'(z)$ exists and coincides with~$f'(\sigma)$, see, \textit{e.g.},~\cite[Proposition~4.7 on p.\,79]{Pommerenke2}. \REM{However, in general $\anglim_{z\to\sigma}f'(z)$ may not exist if $f'(\sigma)=\infty$.}

Below we formulate a version of the classical Julia\,--\,Wolff\,--\,Charath\'eodory Theorem, see, \textit{e.g.}, \cite[Theorem~1.2.5, Proposition~1.2.6, Theorem~1.2.7]{Abate}.
\begin{result}[Julia\,--\,Wolff\,--\,Charath\'eodory]\label{TH_JWC}
Let $\varphi\in\Hol(\D,\D)$ and ${\sigma\in\UC}$. Then the following statements hold:
\begin{mylist}
\item[(i)] If $\varphi(\sigma):=\anglim_{z\to\sigma}\varphi(z)$ exists and belongs to~$\UC$, then the angular derivative\linebreak $\varphi'(\sigma):=\vphantom{\int\limits^1}\anglim_{z\to\sigma}\frac{\varphi(z)-\varphi(\sigma)}{z-\sigma}$ exists, finite or infinite.

    \vskip5mm

\item[(ii)] The following assertions are equivalent:\\
\begin{itemize}
\item[(ii.1)] the angular limit $\varphi(\sigma):=\anglim_{z\to\sigma}\varphi(z)$ exists and belongs to~$\UC$ and the angular derivative $\varphi'(\sigma)$ is finite;
\item[(ii.2)] $\alpha_\varphi(\sigma):=\displaystyle\vphantom{\int\limits_0^1}\liminf_{z\to\sigma}\dfrac{1-|\varphi(z)|}{1-|z|}<+\infty$;
\item[(ii.3)] there exists $A>0$ and $\omega\in\UC$ such that
\begin{equation}\label{EQ_JuliaLemma}
\frac{|\omega-\varphi(z)|^2}{1-|\varphi(z)|^2}\le A\frac{|\sigma-z|^2}{1-|z|^2}\quad\text{for all $z\in\UD$.}
\end{equation}
\end{itemize}\vskip.2ex
\item[(ii)] If assertions (ii.1)\,--\,(ii.3) take place, then $\varphi(\sigma)=\omega$ and ${\varphi'(\sigma)=\omega\,\overline\sigma\alpha_\varphi(\sigma)=\omega\,\overline\sigma A_0}$, where $A_0$ is the minimal value of~$A$ for which~\eqref{EQ_JuliaLemma} holds. In particular, $\overline\omega\,\sigma\varphi'(\sigma)>0$.
\end{mylist}
\end{result}

\begin{definition}\label{DF_contact-BRFP}
If the equivalent conditions (ii.1)\,--\,(ii.3) in the above theorem hold, then the point~$\sigma$ is said to be a \textsl{regular contact point}  of the self-map~$\varphi$. If in addition, $\varphi(\sigma)=\sigma$, then $\sigma$ is called a \textsl{boundary regular fixed point} (abbreviated as BRFP) of~$\varphi$.
\end{definition}

By the classical Denjoy\,--\,Wolff Theorem (see, \textit{e.g.}, \cite[Theorem~1.2.14, Corollary~1.2.16, Theorem~1.3.9]{Abate}), for any $\varphi\in\Hol(\UD,\UD)\setminus\{\id_\UD\}$ there exists a unique~$\tau\in\overline\UD$ with the following property: $\tau$ is either internal fixed point of~$\varphi$, \textit{i.e.} $\tau\in\UD$ and $\varphi(\tau)=\tau$, or $\tau$ is a boundary regular fixed point of~$\varphi$ with $\varphi'(\tau)\le1$. Moreover, if $\varphi$ is not an elliptic automorphism, then the iterates $\varphi^{\circ n}\to\tau$ locally uniformly in~$\UD$ as $n\to+\infty$.
\begin{definition}\label{DF_DW}
In the above notation, $\tau$ is called the \textsl{Denjoy\,--\,Wolff point} (in short, \textsl{DW-point}) of~$\varphi$.
\end{definition}

\begin{remark}\label{RM_JWC}
With the help of conformal mapping,  Theorem~\ref{TH_JWC} can be applied to self-maps of the right half-plane ${\UH:=\{\zeta:\Re\zeta>0\}}$. In this way, for any $f\in\Hol(\UH,\UH)$ inequality~\eqref{EQ_JuliaLemma} leads to
\begin{equation}\label{EQ_Julia_half-plane}
 \Re f(z)\ge f'(\infty)\, \Re z\quad\text{for all~}z\in\UH,
\end{equation}
where $f'(\infty):=\anglim_{z\to\infty}f(z)/z$ exists finitely, with $f'(\infty)\ge0$. Thanks to the Maximum Principle applied to the harmonic function~$z\mapsto\Re f(z)-f'(\infty)\, \Re z$, if equality holds in~\eqref{EQ_Julia_half-plane} at some point~$z\in\UH$, then it holds for all~$z\in\UH$.
\end{remark}

\begin{remark}\label{RM_CowenPomm}
It is known, see, \textit{e.g.}, \cite[p.\,275]{CowenPommerenke}, that if $\sigma_1$ and $\sigma_2$ are BRFP's of a self-map~$\varphi\in\Hol(\UD,\UD)$, then $\varphi'(\sigma_1)\,\varphi'(\sigma_2)\ge1$, with strict inequality sign unless $\varphi$ is an automorphism.
\end{remark}

\subsection{Parametric Representation and evolution families in the disk}\label{SS_param-represent}

A general form of the non-autonomous analogue of equation~\eqref{EQ_autonomousFlow}, suitable for our purpose, was suggested by Bracci, Contreras and D\'\i{}az-Madrigal in their seminal paper~\cite{BCM1}. The role of infinitesimal generators in their equation is played by the so-called \textsl{Herglotz vector fields}.

\begin{definition}[\cite{BCM1}]\label{DF_HVF}
A function $G:\mathbb{D}\times[0,+\infty)\to \mathbb{C}$ is called a \textsl{Herglotz vector field (in the unit disk)} if it satisfies the following conditions:

\begin{mylist}
\item[HVF1.] for every $z\in\mathbb{D}$, the function $[0,+\infty)\ni t\mapsto G(z,t)$ is measurable;

\item[HVF2.] for a.e. $t\in[0,+\infty),$ the function $\mathbb{D}\ni z\mapsto G_t(z):=G(z,t)$ is an infinitesimal generator, \textit{i.e.} $G_t\in\mathrm{T}\mathfrak{U}$;

\item[HVF3.] for any compact set $K\subset\mathbb{D}$ and any $T>0$ there
exists a non-negative locally integrable function $k_{K,T}$ on~$[0,+\infty)$ such that $|G(z,t)|\le k_{K,T}(t)$ for all $z\in K$ and a.e.~$t\in[0,T].$
\end{mylist}
\end{definition}

It is known~\cite[Theorem\,4.4]{BCM1} that every Herglotz vector field~$G$ is \textsl{semicomplete}, {\it i.e.} for any $z\in\UD$ and any~$s\ge0$ the initial value problem
\begin{equation}\label{EQ_general-LK-INI}
\DDT\varphi_{s,t}(z)=G\big(\varphi_{s,t}(z),t\big),\quad \varphi_{s,s}(z)=z,
\end{equation}
has a unique solution $t\mapsto\varphi_{s,t}(z)$ defined for all~$t\ge s$.

Thanks to the Berkson\,--\,Porta formula~\eqref{EQ_BP}, any Herglotz vector field~$G$ can be expressed in the form $G(z,t)=\big(\tau(t)-z\big)\big(1-\overline{\tau(t)}z\big)\,p(z,t)$, where $\tau:[0,+\infty)\to\overline\UD$ and $\Re p\ge0$. In this way, Herglotz vector fields can be characterized via the \textsl{Berkson\,--\,Porta data}~$(p,\tau)$, see \cite[Theorem\,4.8]{BCM1}. In particular, if we set $\tau\equiv0$ and $p(0,\cdot)\equiv1$, then assuming that $p(z,\cdot)$ is measurable for each~$z\in\UD$, from equation~\eqref{EQ_general-LK-INI} we obtain the classical Loewner\,--\,Kufarev ODE~\eqref{EQ_LK-ODE-intro}. It is known that the union of all semiflows, \textit{i.e.} the \textsl{reachable set} of~\eqref{EQ_LK-ODE-intro}, coincides with the class~$\mathfrak{U}_0$ of all univalent~$\varphi\in\Hol(\UD,\UD)$ with $\varphi(0)=0$ and~${\varphi'(0)>0}$.

The latter statement is (one of essentially equivalent) formulation of Loewner's much celebrated \textsl{Parametric Representation} of univalent functions due to Loewner~\cite{Loewner}, Kufarev~\cite{Kufarev1943}, Pommerenke~\cite{Pommerenke-65}, \cite[Chapter\,6]{Pommerenke}, and Gutljanski~\cite{Gutljanski}. Note that $\mathfrak{U}$ is a semigroup w.r.t. the composition operation~$(\psi,\varphi)\mapsto\psi\circ\varphi$. A natural question arises: given a subsemigroup~$\mathfrak{S}\subset\Hol(\UD,\UD)$, is it possible to represent~$\mathfrak{S}$ using the same idea, \textit{i.e.} as the reachable set of a suitable special case of~\eqref{EQ_general-LK-INI}?
More precisely, we introduce the following definition.

For $\mathfrak{S}\subset\Hol(\UD,\UD)$ we denote by $\mathrm{T}\mathfrak{S}$, further on referred to as the \textsl{infinitesimal structure} of~$\mathfrak S$, the set of all infinitesimal generators~$G$ giving rise, via equation~\eqref{EQ_autonomousFlow}, to one-parameter semigroups $(\phi_t^G)\subset\mathfrak S$.

\begin{definition}\label{DF_Loewner-type_representation}
We say that a subsemigroup~$\mathfrak{S}\subset\Hol(\UD,\UD)$ \textsl{admits Loewner-type parametric representation} if there exists a convex cone $\mathcal M_{\mathfrak S}$ of Herglotz vector fields in~$\UD$ with the following properties:
\begin{mylist}
\item[LPR1.] for every $G\in\mathcal M_{\mathfrak S}$, we have $G(\cdot,t)\in \mathrm{T}\mathfrak{S}$ for a.e.~$t\ge0$;

\item[LPR2.] for every $G\in\mathcal M_{\mathfrak S}$, the solution $\varphi_{s,t}$ to the initial value problem~\eqref{EQ_general-LK-INI} satisfies $\varphi_{s,t}\in\mathfrak S$ for any~$s\ge0$ and any~$t\ge s$;

\item[LPR3.] for every $\varphi\in\mathfrak S$ there exists $G\in\mathcal M_{\mathfrak S}$ such that $\varphi=\varphi_{s,t}$ for some $s\ge0$ and~$t\ge s$, where $\varphi_{s,t}$ stands, as above, for the solution to~\eqref{EQ_general-LK-INI}.
\end{mylist}
\end{definition}

The main results of~\cite{BCM1} says that similarly to the theory of one-parameter semigroups, the semiflows of~\eqref{EQ_general-LK-INI} can be characterized in an intrinsic way without appealing to differential equations. This is fact will play a very important role in our argument.

\begin{definition}[\cite{BCM1}]\label{DF_evol_family}
A family $(\varphi_{s,t})_{0\leq s\leq t}$ of
holomorphic self-maps of the unit disk is called an {\sl evolution
family} if it satisfies the following conditions:
\begin{mylist}
\item[EF1.] $\varphi_{s,s}=\id_{\mathbb{D}}$ for any~$s\ge0$;

\item[EF2.] $\varphi_{s,t}=\varphi_{u,t}\circ\varphi_{s,u}$ whenever $0\leq
s\leq u\leq t$;

\item[EF3.] for all $z\in\mathbb{D}$ and  $T>0$ there exists an
integrable function $k_{z,T}:[0,T]\to[0,+\infty)$
such that
\[
|\varphi_{s,u}(z)-\varphi_{s,t}(z)|\leq\int_{u}^{t}k_{z,T}(\xi)\DI\xi
\]
whenever $0\leq s\leq u\leq t\leq T.$
\end{mylist}
\end{definition}
\begin{remark}\label{RM_univalent}
Although it is not required in the above definition, all elements of an evolution family are univalent in~$\UD$, see \cite[Corollary~6.3]{BCM1}.
\end{remark}

In~\cite[Theorem\,1.1]{BCM1} it is proved that for any Herglotz vector field~$G$ the solution~$(\varphi_{s,t})$ to~\eqref{EQ_general-LK-INI} is an evolution family and, conversely, any evolution family~$(\varphi_{s,t})$ satisfies~\eqref{EQ_general-LK-INI} with some Herglotz vector field~$G$. This correspondence is one-to-one if we identify Herglotz vector fields $G(\cdot,t)$ that coincide for a.e.~$t\ge0$.

\begin{definition}\label{DF_associated}
In the above notation, the Herglotz vector field~$G$ and the corresponding evolution family~$(\varphi_{s,t})$ are said to be \textsl{associated} with each other.
\end{definition}

In the next section we will introduce the semigroups~$\mathfrak S$ which we are going to study in this paper. Further, using the above 1-to-1 correspondence, we will reformulate the problem of Loewner-type representation for these semigroups as a \textsl{problem of embeddability} in evolution families.

Let us remark that condition LPR2 in Definition~\ref{DF_Loewner-type_representation} does not follow from~LPR1. Indeed, consider the semigroup~$\mathfrak S:=\UF[\{1\}]$ of all univalent self-maps ${\varphi\in\Hol(\UD,\UD)}$ having a BRFP at~$\sigma=1$.
By \cite[Theorem~1]{CDP2}, an infinitesimal generator~$G$ belongs to $\mathrm{T}\mathfrak{U}[\{1\}]$ if and only if $G$ has a boundary regular null-point at~$\sigma:=1$, \textit{i.e.} $G'(1):=\anglim_{z\to1}G(z)/(z-1)$ exists finitely. At the same time, by~\cite[Theorem~1.1]{BRFPLoewTheory},  a Herglotz vector field $G$ generates an evolution family that lies in $\UF[\{1\}]$ if and only if the following two conditions are met: (a)~$G(\cdot,t)$ has a boundary regular null-point at~$\sigma:=1$ for a.e.~$t\ge0$, and (b)~the function $t\mapsto G'(1,t)$ is locally integrable on~$[0,+\infty)$. Examples given in~\cite[Sect.\,6]{BRFPLoewTheory} show that there are Herglotz vector fields~$G$ satisfying (a) but not (b). Thus, for the convex cone of all Herglotz vector fields satisfying~(a) meets condition~LPR1 but fails to meet condition~LPR2.

It seems plausible to conjecture that LPR2 implies~LPR1, but this interesting question goes slightly out of the framework of the present paper.

We conclude this section with one standard (but not so well-known in this form) result from the classical Loewner Theory, see, \textit{e.g.}, \cite[Problem~3 on p.\,164]{Pommerenke} or \cite[pp.\,69-70]{Aleksandrov}.
\begin{result}\label{TH_standRepr}
Let $\psi\in\Hol(\UD,\UD)\setminus\{\id_\UD\}$ be a univalent function with $\psi(0)=0$ and $\psi'(0)>0$. Then there exists an evolution family~$(\varphi_{s,t})$ such that $\varphi_{0,1}=\psi$ and $\varphi_{s,t}(0)=0$, $\varphi'_{s,t}(0)=\exp\big((t-s)\log\psi'(0)\big)$ whenever $t\ge s\ge0$.
\end{result}

\section{Statement of the problem and its reformulation via evolution families}\label{S_PrStatement}
 Given a set $F\subset\UC$ we denote by $\mathfrak P[F]$ the class of all self-maps $\varphi\in\Hol(\UD,\UD)$ for which every $\sigma\in F$ is a BRFP. Furthermore, for $\tau\in\overline\UD\setminus F$  we denote by $\mathfrak P_\tau[F]$ the subclass of~$\mathfrak P[F]$ that consists of the identity map~$\id_\UD$ and all $\varphi\in \mathfrak P[F]\setminus\{\id_\UD\}$ whose DW-point coincides with~$\tau$.
Finally, by $\UF[F]$ and $\Ut\tau[F]$ we denote the classes formed by all univalent self-maps from $\mathfrak P[F]$ and $\mathfrak P_\tau[F]$, respectively. Note that according to the Chain Rule for angular derivatives (see, \textit{e.g.}, \cite[Lemma~(1.3.25) on p.\,92]{Abate}) these classes are semigroups with identity w.r.t. the operation of composition.

Note that given a self-map $\phi\in\Hol(\UD,\UD)\setminus\{\id_\UD\}$, the set of all BRFPs of~$\phi$ is at most countable. This easily follows from Cowen\,--\,Pommerenke inequalities~\cite[Theorem~4.1]{CowenPommerenke} combined with the fact that $\phi'(\sigma)>1$ for any BRFP~$\sigma$ different from the Denjoy\,--\,Wolff point of~$\phi$.

 Therefore, it does not make sense to consider $\UF[F]$ if $F$ is uncountable. On the other hand, the following example shows that for any \textsl{finite} set~$F\subset\partial\UD$ and any~$\tau\in\overline\UD\setminus F$, the semigroups $\UF[F]$ and $\Ut\tau[F]$ are not trivial.
\begin{example}\label{EX_non-trivial}
Given $n\ge1$ pairwise distinct points $\sigma_1,\ldots,\sigma_n\in\partial\UD$ and $\tau\in\overline\UD\setminus F$, $F:=\{\sigma_1,\ldots,\sigma_n\}$, let us construct a self-map $\phi\neq\id_\UD$ that belongs to~$\Ut\tau[F]$. To this end, consider the rational function
$$
G(z):=(\tau-z)(1-\overline\tau z)/p(z),\quad p(z):=\sum_{j=1}^n\alpha_j\frac{\sigma_j+z}{\sigma_j-z},
$$
where $\alpha_1,\ldots,\alpha_n$ are arbitrary positive coefficients. Since $p$, and hence $1/p$, have positive real part, $G$ is an infinitesimal generator by the Berkson\,--\,Porta formula~\eqref{EQ_BP}. Clearly, the points~$\sigma_1,\ldots,\sigma_n$ are zeros of~$G$. Therefore, by~\cite[Theorem~1]{CDP2}, the one-parameter semigroup~$(\phi_t)$ generated by~$G$ via~\eqref{EQ_autonomousFlow} lies in~$\Ut\tau[F]\subset \UF[F]$ and $G\in \mathrm{T}\Ut\tau[F]\subset  \mathrm{T}\UF[F]$. Note that $\phi_t\neq\id_\UD$ for any~$t>0$ because $G\not\equiv0$.
\end{example}

The main problem we address in this paper is the following.

\vskip1ex
\noindent{\bf PROBLEM.} {\it
Given a finite set $F\subset\partial\UD$ and $\tau\in\overline\UD\setminus F$, do the semigroups $\UF[F]$ and $\Ut\tau[F]$  admit Loewner-type parametric representation in the sense of Definition~\ref{DF_Loewner-type_representation}?}
\vskip1ex

We are able to give an affirmative answer for a part of the cases, see Theorem~\ref{TH_main} in the Introduction.
The first step is to reduce the above problem to a problem of embeddability in evolution families.
\begin{theorem}\label{TH_embedd}
Let $\mathfrak S:=\UF[F]$ or $\mathfrak S:=\Ut\tau[F]$ with some finite set~$F\subset\partial\UD$ and $\tau\in\overline\UD\setminus F$. Then the following conditions are equivalent:
\begin{itemize}
\item[(i)] $\mathfrak S$ admits Loewner-type parametric representation;
\item[(ii)] for any $\varphi\in\mathfrak S\setminus\{\id_\UD\}$ there exists an evolution family $(\psi_{s,t})\subset\mathfrak S$ such that $\varphi=\psi_{0,1}$.
\end{itemize}

\begin{proof}[Proof of Theorem~\ref{TH_embedd}]
Note that (i) implies (ii) for any semigroup~$\mathfrak S$. Indeed, assume that~(i) holds. Then by Definition~\ref{DF_Loewner-type_representation}, for any $\varphi\in\mathfrak S\setminus\{\id_\UD\}$ there exists a Herglotz vector field $G$ with associated evolution family~$(\varphi_{s,t})\subset\mathfrak S$ such that $\varphi=\varphi_{s_0,t_0}$ for some $s_0\ge0$ and~$t_0>s_0$. Setting ${\psi_{s,t}:=\varphi_{s_0+qs,s_0+qt}}$, $q:=t_0-s_0$, for all~${s\ge0}$ and all~$t\ge s$ proves~(ii).

Now we assume that~(ii) holds. Consider first the case ${\mathfrak S=\UF[F]}$. To show that (i) holds, define $\mathcal M_{\mathfrak S}$ to be the set of all Herglotz vector fields $G$ with the following properties:
\begin{itemize}
\item[(a)] for each~$\sigma\in F$ and a.e.~$t\ge0$ there exists a finite angular limit
$
\lambda_\sigma(t):=\anglim_{z\to\sigma}G(z,t)/(z-\sigma);
$
\item[(b)] the functions $\lambda_\sigma$, $\sigma\in F$, are locally integrable on~$[0,+\infty)$.
\end{itemize}
Then thanks to the fact that the set of all infinitesimal generators is a convex cone, see, \textit{e.g.}, \cite[Theorem\,1.4.15]{Abate}, the set $\mathcal M_{\mathfrak S}$ is also a convex cone.
Moreover, according to \cite[Theorem\,1]{CDP2} and \cite[Theorem\,1.1]{BRFPLoewTheory}, $\mathcal M_{\mathfrak S}$ satisfies conditions LPR1 and~LPR2 in Definition~\ref{DF_Loewner-type_representation}.

Now let $\varphi\in\UF[F]$. Then by~(ii) there exists an evolution family ${(\psi_{s,t})\subset \UF[F]}$ such that $\psi_{0,1}=\varphi$. By \cite[Theorem\,1.1]{BRFPLoewTheory}, the vector field~$G$ associated with~$(\psi_{s,t})$ satisfies conditions (a) and~(b). This proves~LPR3.

It remains to explain how the proof of (ii)~$\Rightarrow$~(i) should be modified in the case~${\mathfrak S=\Ut\tau[F]}$. First of all, to conditions (a) and~(b) we should add condition
\begin{itemize}
\item[(c)] $G(z,t)=(\tau-z)(1-\overline\tau z)p(z,t)$ for all~$z\in\UD$, a.e.~$t\ge0$ and some~$p:\UD\times[0,+\infty)\to\C$ with~$\Re p\ge0$.
\end{itemize}
To make sure that LPR1 and LPR2 hold we should additionally use \cite[Theorem\,1.4.19]{Abate} and \cite[Corollary\,7.2]{BCM1}, respectively. Finally, to see that the vector field~$G$ associated with~$(\psi_{s,t})$ satisfies condition~(c), one should additionally use \cite[Theorem~6.7]{BCM1}.

The proof is now complete.
\end{proof}

\end{theorem}

\section{Evolution families with boundary regular fixed points}\label{S_diff_theorems}
In this section we study conditions under which a family $(\varphi_{s,t})_{t\ge s\ge 0}\subset\mathfrak S$, where ${\mathfrak S:=\mathfrak P[F]}$ or ${\mathfrak S:=\mathfrak P_\tau[F]}$, satisfying conditions EF1 and~EF2 from Definition~\ref{DF_evol_family}, is an \textsl{evolution family}.

\begin{remark}\label{RM_Chain-Rule}
In the above notation, by the Chain Rule for angular derivatives, see, \textit{e.g.}, \cite[Lemma~(1.3.25) on p.\,92]{Abate} or \cite[Lemma~2]{CDP2}, for any $\sigma\in F$ (and also for $\sigma:=\tau$ in case $\mathfrak S=\mathfrak P_\tau[F]$) we have $\varphi_{s,t}'(\sigma)=\varphi_{s,u}'(\sigma)\varphi'_{u,t}(\sigma)$ whenever $t\ge u\ge s\ge0$.
\end{remark}

\begin{theorem}\label{TH_differentiability}
Let $\tau\in\overline\UD$, $\sigma\in\UC\setminus\{\tau\}$ and let $(\varphi_{s,t})_{t\ge s\ge 0}$ be a family in $\mathfrak P_\tau[\{\sigma\}]$ satisfying conditions EF1 and EF2 from Definition~\ref{DF_evol_family}. Then the following two assertions are equivalent:
\begin{mylist}
\item[(i)] $(\varphi_{s,t})$ is an evolution family;
\item[(ii)] the function $[0,+\infty)\ni t\mapsto \varphi_{0,t}'(\sigma)\in(0,+\infty)$ is locally absolutely continuous.
\end{mylist}
\end{theorem}
\begin{proof}
The fact that (i) implies~(ii) follows readily from \cite[Theorem~1.1]{BRFPLoewTheory}.
The converse implication for $\tau\in\UD$ is proved in~\cite{GoryainovDiff}, but we sketch here a bit different version of that proof. First of all, in case $\tau\in\UD$, using automorphisms of~$\UD$ we may assume that~$\tau=0$. Then applying Theorem~\ref{TH_JWC} for $\varphi(z):=\varphi_{u,v}(z)/z$ with $v\ge u\ge0$ and using Remark~\ref{RM_Chain-Rule}, it is easy to see that (ii) implies that $[0,+\infty)\ni t\mapsto\varphi_{0,t}'(0)\in\Complex$ is also locally absolutely continuous and that $\varphi_{0,t}'(0)\neq0$ for all~$t\ge0$. Therefore, (i) follows from \cite[Theorem~7.3]{BCM1} (or, equivalently, from \cite[Proposition~2.10]{SMP}).

Now assume that (ii) holds and that $\tau\in\UC$. Using the fact that $\varphi'_{s,t}(\tau)\le 1$ but $\varphi'_{s,t}(\tau)\varphi'_{s,t}(\sigma)\ge1$ for all $s\ge0$ and $t\ge s$ by Remark~\ref{RM_CowenPomm}, and bearing in mind Remark~\ref{RM_Chain-Rule}, we see that $[0,+\infty)\ni t\mapsto \varphi_{0,t}'(\tau)\in(0,+\infty)$ is locally absolutely continuous. The rest of the proof is reduced to the proposition below.
\end{proof}

\begin{proposition}\label{PR_differentiability_2points}
Let $\sigma_1,\sigma_2$ be two distinct points of~$\UC$ and let $(\varphi_{s,t})_{t\ge s\ge 0}$ be a family in~$\mathfrak P[\{\sigma_1,\sigma_2\}]$ satisfying  conditions EF1 and EF2 from Definition~\ref{DF_evol_family}. Then the following two assertions are equivalent:
\begin{mylist}
\item[(i)] $(\varphi_{s,t})$ is an evolution family;
\item[(ii)] the functions $[0,+\infty)\ni t\mapsto \varphi_{0,t}'(\sigma_j)\in(0,+\infty)$, $j=1,2$, are locally absolutely continuous.
\end{mylist}
\end{proposition}
\begin{proof}
As before, the implication (i) $\Longrightarrow$~(ii) is by \cite[Theorem~1.1]{BRFPLoewTheory}.

So let us assume (ii) and prove~(i).
Let $H$ be any conformal map of $\UD$ onto the right half-plane~$\UH$ with $H(\sigma_1)=\infty$, $H(\sigma_2)=0$.  Fix any $T>0$. For all $u,t\in[0,T]$ with $u\le t$ we define $g_{u,t}(\zeta):=H(\varphi_{u,t}(\zeta))-H(\zeta)/\varphi'_{u,t}(\sigma_1)$. By Remark~\ref{RM_JWC}  applied to $H\circ\varphi_{u,t}\circ H^{-1}$,  either $g_{u,t}\equiv i\,C$ for some constant $C\in\Real$ or $g_{u,t}\in\Hol(\UD,\UH)$ and hence $f_{u,t}(w):=1/g_{u,t}\big( H^{-1}(1/w))$ is a holomorphic self-map of~$\UH$. Since $\sigma_2$ is a BRFP of~$\varphi_{u,t}$, in the former case we get $g_{u,t}\equiv0$, while in the latter case we have  $$f'_{u,t}(\infty)=\frac{\varphi'_{u,t}(\sigma_1)}{\varphi'_{u,t}(\sigma_1)\varphi'_{u,t}(\sigma_2)-1}\in(0,+\infty)$$ and, by~\eqref{EQ_Julia_half-plane} for $f:=f_{u,t}$,
$$
\left|\frac{1}{g_{u,t}(\zeta)}\right|\ge\Re \frac{1}{g_{u,t}(\zeta)}\ge f'_{u,t}(\infty)\,\Re\frac{1}{H(\zeta)}\quad\lefteqn{\text{for all~$\zeta\in\UD$}.}
$$
Therefore, in both cases, for any~$\zeta\in\UD$,
$$
\big|\varphi_{u,t}(\zeta)-\zeta\big|\le 2\big|H(\varphi_{u,t}(\zeta))-H(\zeta)\big|\le 2\Big(\big|g_{u,t}(\zeta)\big| + \left|H(\zeta)\Big(1-\frac1{\varphi'_{u,t}(\sigma_1)}\Big)\right|\Big).
$$

In view of Remark~\ref{RM_Chain-Rule}, it follows that for any $K\subset\subset\UD$ there exists $M_K>0$ such that
\begin{equation}\label{EQ_varphi(zeta)-zeta}
|\varphi_{u,t}(\zeta)-\zeta|\le M_K\Big(\big|1-1/\varphi'_{u,t}(\sigma_1)\big|+\big|1-\varphi'_{u,t}(\sigma_2)\big|\Big)
\le M_K\,\sum_{j=1}^{2}\frac{\big|\varphi'_{0,t}(\sigma_j)-\varphi'_{0,u}(\sigma_j)\big|} {\min\limits_{0\le v\le T}\varphi'_{0,v}(\sigma_j)}
\end{equation}
for all $\zeta\in K$ and any $u,t\in[0,T]$ with $u\le t$.

Now fix any $z\in\UD$. By~(ii) and Remark~\ref{RM_Chain-Rule}, $$\sup\{\varphi_{s,u}'(\sigma_j):\,0\le s\le u\le T,\,j=1,2\}<+\infty.$$  Hence, using Theorem~\ref{TH_JWC} for $\varphi_{s,u}$ twice, at~$\sigma_1$ and at~$\sigma_2$, we conclude that $$K(z):=\{\varphi_{s,u}(z):0\le s\le u\le T\}\subset\subset \UD.$$ Applying now~\eqref{EQ_varphi(zeta)-zeta} with $K:=K(z)$ and $\zeta:=\varphi_{s,u}(z)$, where $0\le s\le u\le T$, we get
$$
|\varphi_{s,t}(z)-\varphi_{s,u}(z)|\le \int_u^t k_{z,T}(\xi)\DI\xi,\quad k_{z,T}(\xi):=M_{K(z)}\,\sum_{j=1}^{2} \frac{\big|\frac{\di}{\di\xi}\,\varphi'_{0,\xi}(\sigma_j)\big|} {\min\limits_{0\le v\le T}\varphi'_{0,v}(\sigma_j)},
$$
whenever $0\le s\le u\le t\le T$. This proves condition EF3, as it was desired.
\end{proof}

\begin{remark}
Fix some~$z_0\in\UD$. Using Theorem~\ref{TH_JWC} it is not difficult to show that assertion~(ii) in Proposition~\ref{PR_differentiability_2points} is equivalent to:
\begin{mylist}
\item[(iii)] $[0,+\infty)\ni t\mapsto \varphi'_{0,t}(\sigma_1)\varphi'_{0,t}(\sigma_2)$ and $[0,+\infty)\ni t\mapsto \varphi_{0,t}(z_0)$ are locally absolutely continuous.
\end{mylist}
\end{remark}

It turns out that in case of \textsl{more than two} boundary regular fixed points even the regularity requirement concerning $t\mapsto\varphi_{0,t}(z_0)$  may be omitted.

\begin{proposition}
Let $F:=\{\sigma_1,\sigma_2,\ldots\sigma_n\}\subset \UC$, where $\sigma_j$'s are pairwise distinct points and $n\ge3$. Let  $(\varphi_{s,t})_{t\ge s\ge 0}$ be a family in~$\mathfrak P[F]$ satisfying  conditions EF1 and EF2 from Definition~\ref{DF_evol_family}. Then the following two assertions are equivalent:
\begin{mylist}
\item[(i)] $(\varphi_{s,t})$ is an evolution family;
\item[(ii)] the function $[0,+\infty)\ni t\mapsto \prod_{j=1}^n\varphi_{0,t}'(\sigma_j)\in(0,+\infty)$ is locally absolutely continuous.
\end{mylist}
\end{proposition}
\begin{proof}
Again, in view of~\cite[Theorem~1.1]{BRFPLoewTheory}, we only need to prove that (ii) implies (i).
For $j=1,\ldots,n$ and $t\ge0$ denote $\lambda_j(t):=\log\varphi_{0,t}(\sigma_j)$ and $\lambda(t):=\sum_{j=1}^n\lambda_j(t)$. Fix now any $s\ge0$ and any $t\ge s$. Trivially, $\lambda_j(t)-\lambda_j(s)=0$ for all $j=1,\ldots,n$ if $\varphi_{s,t}=\id_\UD$. So suppose that $\varphi_{s,t}\neq\id_\UD$ and consider the following two cases.

\case1{The Denjoy\,--\,Wolff point of $\varphi_{s,t}$ does not belong to $F$.}
Then, in view of Remark~\ref{RM_Chain-Rule},\\\hbox to\textwidth{\hss $\lambda_j(t)-\lambda_j(s)>0$  and hence $0<\lambda_j(t)-\lambda_j(s)<\lambda(t)-\lambda(s)$ for all $j=1,\ldots,n$.}

\case2{One of the points $\sigma_1,\ldots,\sigma_n$ is the Denjoy\,--\,Wolff point of $\varphi_{s,t}$.} Without loss of generality we may assume that this point is~$\sigma_1$.
Fix any natural numbers ${j\in[2,n]}$ and $k\in{[2,n]\setminus\{j\}}$. Recall that  $\varphi'_{s,t}(\sigma_k)\varphi'_{s,t}(\sigma_1)>1$ by Remark~\ref{RM_CowenPomm}. Therefore, again we have ${0<\lambda_j(t)-\lambda_j(s)<\lambda(t)-\lambda(s)}$ for all  $j=2,\ldots,n$. It follows, in addition, that $$0\ge \lambda_1(t)-\lambda_1(s)=\lambda(t)-\lambda(s)-\sum_{j=2}^n\big(\lambda_j(t)-\lambda_j(s)\big)\ge -(n-2)(\lambda(t)-\lambda(s)).$$

Thus if (ii) holds, then all $\lambda_j$'s are locally absolutely continuous on~$[0,+\infty)$. By Proposition~\ref{PR_differentiability_2points} the latter implies that $(\varphi_{s,t})$ is an evolution family, which was to be proved.
\end{proof}

\REM{\begin{remark}
Taking into account the above statements, it might seem natural to put forward the following conjecture: if $(\varphi_{s,t})_{0\le s\le t}\subset\mathfrak P[\{\sigma_1,\sigma_2,\sigma_3\}]$, where $\sigma_j$'s are pairwise distinct points on~$\UC$, satisfies conditions EF1 and EF2 and if $[0,+\infty)\ni t\mapsto \varphi'_{0,t}(\sigma_1)$ is locally absolutely continuous, then $(\varphi_{s,t})$ is an evolution family.
\end{remark}}

\section{Proof of the main results}\label{S_PROOF}
In this section we will prove Theorem~\ref{TH_main} and Corollary~\ref{CR_specific}.
Throughout the whole section we will assume that $\tau,\sigma_1,\sigma_2,\ldots,\sigma_n$ are pairwise distinct points, $\tau\in\overline\UD$ and $F:=\{\sigma_1,\sigma_2,\ldots,\sigma_n\}\subset\UC$.

\subsection{Lemmas}\label{SS_lemmas}
First of all we prove a version of the Chain Rule for {\sl finite} angular derivatives.
\begin{lemma}\label{LM_ChainRule}
Let $\psi_1,\psi_2\in\Hol(\UD,\UD)$. If $\sigma\in\UC$ is a regular contact point of $\psi_2\circ\psi_1$, then $\sigma$ is also a regular contact point of~$\psi_1$ and $\psi_1(\sigma)$ is a regular contact point of~$\psi_2$, with $(\psi_2\circ\psi_1)'(\sigma)=\psi_2'(\psi_1(\sigma))\psi_1'(\sigma)$.
\end{lemma}
\begin{proof}
According to Theorem~\ref{TH_JWC}, there exists a sequence $(z_n)\subset\UD$  converging to~$\sigma$ such that $$\liminf_{n\to+\infty}\frac{1-|\psi_2(\psi_1(z_n))|}{1-|z_n|}<+\infty.$$
Passing to a subsequence, we may assume that $\big(\psi_1(z_n)\big)$ also converges. It is easy to see that $\omega:=\lim_{n\to+\infty}\psi_1(z_n)\in\UC$. Taking into account that by Theorem~\ref{TH_JWC}, both $\alpha_{\psi_1}(\sigma)$ and $\alpha_{\psi_2}(\omega)$ are different from zero, we  therefore conclude that
$$\liminf_{n\to+\infty}\frac{1-|\psi_1(z_n)|}{1-|z_n|}\, \liminf_{n\to+\infty}\frac{1-|\psi_2(\psi_1(z_n))|}{1-|\psi_1(z_n)|}\le  \liminf_{n\to+\infty}\frac{1-|\psi_2(\psi_1(z_n))|}{1-|z_n|}<+\infty$$
and hence $\alpha_{\psi_1}(\sigma)$ and $\alpha_{\psi_2}(\omega)$ are finite. Thus $\sigma$ and $\omega$ are regular contact points of $\psi_1$ and~$\psi_2$, respectively, while the formula for $(\psi_2\circ\psi_1)'(\sigma)$ holds by Remark~\ref{RM_Chain-Rule}.
\end{proof}

The following lemma is a slight extension of Theorem~\ref{TH_standRepr}.
\begin{lemma}\label{LM_embedd}
For any univalent $\phi\in\Hol(\UD,\UD)\setminus\Aut(\UD)$ there exists an evolution family $(\varphi_{s,t})$ such that $\phi=\varphi_{0,1}$. Moreover,  $\varphi_{s,t}\not\in\Aut(\UD)$~~ whenever $t>s\ge0$.
\end{lemma}
\begin{proof}
Denote $$h(z):=\frac{|\phi'(0)|}{\phi'(0)}\,\frac{z-\phi(0)}{1-z\overline{\phi(0)}}\quad\text{ for all~$z\in\UD$}.$$ Then $\tilde\phi:=h\circ\phi\in\Hol(\UD,\UD)$ satisfies $\tilde\phi(0)=0$ and $\tilde\phi'(0)>0$ and hence by Theorem~\ref{TH_standRepr} there exists an evolution family $(\tilde\varphi_{s,t})$  such that $\tilde\varphi_{0,1}=\tilde\phi$. Since $h\in\Aut(\UD)$, there exists a one-parameter group~$(h_t)_{t\in\Real}\subset\Aut(\UD)$ such that $h_1=h$. Thus, according to \cite[Lemma~2.8]{SMP}, the functions $\varphi_{s,t}:=h_t^{-1}\circ\tilde\varphi_{s,t}\circ h_s$ form an evolution family, with $\varphi_{0,1}=\phi$ by the very construction. Since according to Theorem~\ref{TH_standRepr}, all $\tilde\varphi_{s,t}$'s with $s\neq t$ are in~$\Hol(\UD,\UD)\setminus\Aut(\UD)$, the same property holds for~$\varphi_{s,t}$'s, and the proof is finished.
\end{proof}

In the most simple cases, the above lemma is enough for our purposes, while in other cases a bit deeper analysis is needed. Denote by $R(D,w_0)$ the conformal radius of a simply connected domain~$D\subset\Complex$ w.r.t.~$w_0$.
\begin{lemma}\label{LM_two_chains}
Let $D_1\subset D_2$, $D'_1\subset D'_2$ be four hyperbolic  simply connected domains such that $D_1\subset D'_1$ and $D_2\setminus D_1=D'_2\setminus D'_1$. Then for any $w_0\in D_1$ we have
\begin{equation}\label{EQ_confRad_ineq}
\frac{R(D_2,w_0)}{R(D_1,w_0)}\le \frac{R(D'_2,w_0)}{R(D'_1,w_0)}
\end{equation}
and the equality holds if only if~$D_1=D_2$.
\end{lemma}
\begin{proof}
Without loss of generality we may assume that~${w_0=0}$. Denote by $K_1$ and $K_2$ the images of $\ComplexE\setminus D_2$ and $\ComplexE\setminus D'_1$, respectively, under the inversion $z\mapsto 1/z$. The hypothesis of the lemma  implies that~$D_2\cap D'_1=D_1$ and $D_2\cup D'_1=D'_2$. Therefore, inequality~\eqref{EQ_confRad_ineq} is equivalent to the following analogue of strong subadditivity for the logarithmic capacity: $$\Cap(K_1\cup K_2)\Cap(K_1\cap K_2) \le \Cap(K_1)\Cap(K_2),$$ which was proved in~\cite{Renggli}.
\end{proof}
\begin{remark}
Inequality~\eqref{EQ_confRad_ineq} leads to a conclusion, which seems to be paradoxical at the first glance: adding the same ``piece'' $D_2\setminus D_1$ to a larger domain~$D'_1$ results in \textsl{larger} relative change of the conformal radius.
\end{remark}

To simplify the statement of the next lemma let us setup some notation. First of all we will assume that the points $\sigma_j$'s are numbered in such a way that the open arc $L_j\subset\UC$ going in the counter-clockwise direction from $\sigma_j$ to~$\sigma_{j+1}$, where for convenience we put $\sigma_{n+1}:=\sigma_1$, does not contain other points from~$F$. Now let $\phi\in\mathfrak P[F]$ and  $f\in\Hol(\UD,\UD)$ be both univalent, with~$f(\UD)\supset \phi(\UD)$. Then for each~$j=1,\ldots,n$ there exists a unique regular contact point~$\xi_j$ of~$f$ such that $f(\xi_j)=\sigma_j$, see, \textit{e.g.}, \cite[Theorem~4.14 on p.\,83]{Pommerenke2}. We denote this point by~$f^{-1}(\sigma_j)$. Further on, we use the lower index, \textit{e.g.}, $a_1$ or $b_3$, to denote the components of vectors $a,a',b,b',x\in[0,1]^n$.
\begin{lemma}\label{LM_multiparam_family}
Let $\phi\in\UF[F]$, $n>1$, and  $z_0\in\UD$. Suppose that there exists no ${j\in[1,n]\cap \Natural}$ such that $\phi$ extends continuously to~$L_j$ with $\phi(L_j)\subset\UC$. Then there exists a family~$(f_a)_{a\in[0,1]^n}$ of  univalent holomorphic self-maps of~$\UD$ satisfying the following conditions:
\begin{mylist}
\item[(i)] $f_a(z_0)=\phi(z_0)$ and $f'_a(z_0)\overline{\phi'(z_0)}>0$ for all $a\in[0,1]^n$;

\item[(ii)] $f_{(0,\ldots,0)}=\phi$ and $f_{(1,\ldots,1)}\in\Aut(\UD)$;

\item[(iii)] if $a,b\in[0,1]^n$, $a\neq b$ and $a_j\le b_j$ for all $j=1,\ldots,n$, then $f_a(\UD)\varsubsetneq f_b(\UD)$;

\item[(iv)] if $k\in[1,n]\cap\Natural$ and if $a,b\in[0,1]^n$ with $a_k\le b_k$ and $a_j=b_j$ for all $j\in([1,n]\setminus\{k\})\cap\N$, then $\varphi_{a,b}:=f_b^{-1}\circ f_a$ extends continuously to $\UC\setminus C_k(a)$, with $\varphi_{a,b}(\UC\setminus C_k(a))=\UC\setminus C_k(b)$, where  $C_k(x)$, $x\in[0,1]^n$, is the closed arc of~$\UC$ going counter-clockwise from~$f_x^{-1}(\sigma_k)$ to~$f^{-1}_x(\sigma_{k+1})$;

\item[(v)] the map~$[0,1]^n\ni a\mapsto R\big(f_a(\UD),\phi(z_0)\big)$ is Lipschitz continuous.
\end{mylist}
\end{lemma}
\begin{proof}
For the sake of simplicity, using automorphisms of~$\UD$, we may assume that~$z_0=0$.
Let $\Gamma:=\bigcup_{j=1}^n[0,\sigma_j)$. For each $j=1,\ldots,n$ denote by~$\Delta_j$ the Jordan domain bounded by $L_j\cup[0,\sigma_j]\cup[0,\sigma_{j+1}]$ and by $\Omega_j$ the Jordan domain bounded by $L_j\cup\phi([0,\sigma_j])\cup\phi([0,\sigma_{j+1}])$. Then $\Delta_j$'s are connected components of ${\UD\setminus\Gamma}$, $\phi(\Delta_j)$'s are connected components of $\phi(\UD)\setminus\phi(\Gamma)$, and $\Omega_j$'s are connected components of ${\UD\setminus \phi(\Gamma)}$. Therefore, $$\UD=\phi(\Gamma)\cup\bigcup_{j=1}^n\Omega_j,\quad \phi(\UD)=\phi(\Gamma)\cup\bigcup_{j=1}^n\phi(\Delta_j),$$~and $\phi(\Delta_j)=\phi(\UD)\cap\Omega_j$ for all~$j=1,\ldots,n$.
In particular, the following statement hold:
\claim{(A)}{Let $D_j$, $j=1,\ldots,n$, be simply connected domains such that $\phi(\Delta_j)\subset D_j\subset\Omega_j$ for each $j=1,\ldots,n$. Then $D:=\phi(\Gamma)\cup\bigcup_{j=1}^n D_j$ is a simply connected domain in~$\UD$.}
Indeed, $D=\phi(\UD)\cup\bigcup_{j=1}^n D_j$, $D_j$ are pairwise disjoint, and $D_j\cap \phi(\UD)=\phi(\Delta_j)$ is simply connected for each~$j=1,\ldots,n$.

Fix any $k\in[1,n]\cap \Natural$.
By~(A) the set $$U_k:=\phi(\Gamma)\cup\phi(\Delta_k)\cup\!\!\bigcup_{\substack{j=1,\ldots,n\\j\neq k}}\!\!\Omega_j$$ is a simply connected domain containing $\phi(\Delta)$. Let $\phi_k$ be any conformal mapping of $\UD$ onto~$U_k$. From the hypothesis of the lemma it follows with the help of the Carath\'eodory Extension Theorem  that $\phi(\Delta_k)\neq\Omega_k$ and hence $U_k\neq\UD$. Then by Lemma~\ref{LM_embedd} there exists an evolution family~$(\varphi^k_{s,t})$ such that $\varphi^k_{0,1}=\phi_k$ and $\varphi^k_{s,t}\not\in\Aut(\UD)$ whenever ${t>s\ge0}$. The family of domains $U_k(s):=\varphi^k_{s,1}(\UD)$, $s\in[0,1]$, has the following properties:

\claim{(B)}{$U_k(0)=U_k$ and $U_k(1)=\UD$;}
\claim{(C)}{$U_k(s)\varsubsetneq U_k(t)$ whenever $0\le s<t\le1;$}
\claim{(D)}{$[0,1]\ni s\mapsto R(U_k(s),\phi(z_0))$ is continuous and strictly increasing.}

In particular, there exists an increasing injective function $s_k:[0,1]\to[0,1]$  such that
$[0,1]\ni\nu\mapsto\log R\big(U_k(s_k(\nu)),0\big)$ is a (non-homogeneous) linear function.
Now we define $$G_a:=\bigcap_{j=1}^n U_j(s_j(a_j))=\phi(\Gamma)\cup\bigcup_{j=1}^n U_j(s_j(a_j))\cap\Omega_j\quad \text{for all~}a\in[0,1]^n.$$

For each $j=1,\ldots,n$ and all $s\in[0,1]$,  $\Omega_j$ is a Jordan domain and $U_j(s)\cap\partial\Omega_j=\phi([0,\sigma_j)\cup[0,\sigma_{j+1}))$ is a cross-cut in~$U_j(s)$ thanks to~(B) and~(C). Hence $U_j(s)\cap\Omega_j\supset \phi(\Delta_j)$ is a simply connected domain. Then by~(A), for all~$a\in[0,1]^n$, $G_a$ is a simply connected domain containing~$\phi(\UD)$. Therefore, for each~$a\in[0,1]^n$ there exists a unique  conformal mapping~$f_a$ of~$\UD$ onto~$G_a$ normalized by~$f_a(0)=\phi(0)$, $f'_a(0)\overline{\phi'(0)}>0$. Clearly, the family~$(f_a)$ defined in this way satisfies conditions (i)\,--\,(iii).

To prove~(iv) we note first that for any $x\in[0,1]^n$, the arc $J_k:=\phi\big([0,\sigma_k)\cup[0,\sigma_{k+1})\big)$ is a cross-cut in~$G_x$. On the one hand, the two connected components of~$G_x\setminus J_k$ are $G_x\cap\Omega_k$ and $G_x\setminus\overline{\Omega_k}$. On the other hand, the preimages of these components under~$f_x$ are the Jordan domains $W'_k(x),W''_k(x)\subset\UD$ bounded by $f_x^{-1}(J_k)\cup\{f_x^{-1}(\sigma_k),f_x^{-1}(\sigma_{k+1})\}\cup C_k(x)$ and $f_x^{-1}(J_k)\cup\{f_x^{-1}(\sigma_k),f_x^{-1}(\sigma_{k+1})\}\cup \big(\UC\setminus C_k(x)\big)$, respectively, see, \textit{e.g.}, \cite[\S2.4]{Pommerenke2}. Moreover,  if $a$ and $b$ are such as in~(iii), then by construction, $G_a\setminus\overline{\Omega_k}=G_b\setminus\overline{\Omega_k}$ and hence $\varphi_{a,b}(W''_k(a))=W''_k(b)$. Then by the Carath\'eodory Extension Theorem, $\varphi_{a,b}|_{W''_k(a)}$ extends continuously and injectively to~$\overline{W''_k(a)}$. Note also that by construction, $\varphi_{a,b}(f_a^{-1}(J_k))=f_b^{-1}(J_k)$. It follows that $\varphi_{a,b}(\UC\setminus C_k(a))=\UC\setminus C_k(b)$, which was to be shown.

It remains to prove~(v). Note first that for each $k\in[0,n]\cap\Natural$, by construction the function $V(a):=\log R(G_a,0)$ is linear  w.r.t.~$a_k$ on the set $$\big\{a\in[0,1]^n:\, a_j=1\text{~for all $j=1,\ldots,n$ except for~$j=k$}\big\}.$$ Suppose now that ${a,b\in[0,1]^n}$, $a_j=b_j$ for all $j=1,\ldots,n$ except for~$j=k$, and $a_k\le b_k$. Define $a',b'\in[1,n]^n$ by setting $a'_k:=a_k$, $b'_k:=b_k$, and $a_j:=b_j:=1$ for all $j=1,\ldots,n$ except for~$j=k$. Then by construction, $G_a\subset G_b$, $G_{a'}\subset G_{b'}$, $G_a\subset G_{a'}$ and $G_b\setminus G_a=G_{b'}\setminus G_{a'}$. Therefore, by Lemma~\ref{LM_two_chains}, $0\le V(b)-V(a)\le V(b')-V(a')$. Thus, $V$ is Lipschitz continuous w.r.t. all~$a_k$'s on the whole set~$[0,1]^n$, and~(v) follows immediately.
\end{proof}

We will make use of the following classical result, known as Loewner's Lemma.
\begin{result}[{Loewner~\cite[Hilfssatz~I]{Loewner}}]\label{TH_LM-Loewner}
If $\varphi\in\Hol(\UD,\UD)$, $\varphi(0)=0$, extends continuously to an open arc~$L\subset\partial\UD$ and $\varphi(L)\subset\partial\UD$, then $|L|\le|\varphi(L)|$, where $|\cdot|$ stands for the length of an arc. The equality holds only if $\varphi\in\Aut(\UD)$.
\end{result}

The statement below can be called a ``boundary three-point version of Loewner's Lemma".
\begin{lemma}\label{LM_3p_Loewner}
Let $f\in\Hol(\UD,\UD)$ be a univalent function with three pairwise distinct BRFPs $\sigma_1,\sigma_2,\tau\in\UC$.  The following statements hold:
\begin{mylist}
\item[(i)]  Suppose that $f$ extends continuously to the open arc $L\subset\UC$ between $\sigma_1$ and $\sigma_2$ that contains~$\tau$, with $f(L)\subset\UC$. Then $f'(\tau)\le1$, with $f'(\tau)=1$ if and only if~$f=\id_\UD$.
\item[(ii)] Suppose that $f$ extends continuously to the open arc $L':=\UC\setminus \overline L$, with $f(L')\subset\UC$. Then $f'(\tau)\ge1$, with $f'(\tau)=1$ if and only if~$f=\id_\UD$.
\end{mylist}
\end{lemma}
\begin{proof}[Proof of (i)]
Consider $\Phi:=H_\tau\circ f\circ H_\tau^{-1}$, where $H_\tau(\zeta):=i(\tau+\zeta)/(\tau-\zeta)$ maps~$\UD$ conformally onto~$\UH_i:=\{z:\Im z>0\}$ with $H_\tau(\tau)=\infty$. Let $\xi_1<\xi_2$ be the images of $\sigma_1$ and~$\sigma_2$ w.r.t.~$H_\tau$. According to the Schwarz Reflection Principle, $\Phi$ extends to a univalent meromorphic function on~$\ComplexE\setminus[\xi_1,\xi_2]$ with a unique pole located at~$\infty$. In particular, it follows that $\infty,\xi_1,\xi_2\not\in \Phi(\Real\setminus[\xi_1,\xi_2])$ and that $\Phi'(x)>0$ for any $x\in\Real\setminus[\xi_1,\xi_2]$.

Therefore, $\Phi(x_2)-\Phi(x_1)>\xi_2-\xi_1$ for all $x_1<\xi_1$ and~$x_2>\xi_2$.
Now write the Nevanlinna Representation Formula for $\Phi$, see, \textit{e.g.}, \cite[p.\,135--142]{Bhatia}:
\begin{equation}\label{EQ_Nevanlinna}
\Phi(z)=\alpha+\beta z+\int_{\xi_1}^{\xi_2}\frac{1+tz}{t-z}\,\DI\nu(t)\quad\text{for all~$z\in\UH_i$,}
\end{equation}
where $\alpha\in\Real$, $\beta:=1/f'(\tau)\ge0$ and $\nu$ is a finite positive Borel measure on $[\xi_1,\xi_2]$. By the uniqueness of the holomorphic extension,~\eqref{EQ_Nevanlinna} holds also for all $z\in\Complex\setminus[\xi_1,\xi_2]$. Hence for all $x_1<\xi_1$ and~$x_2>\xi_2$ we have
$$
\Phi(x_2)-\Phi(x_1)=(x_2-x_1)\Big[\beta-\int_{\xi_1}^{\xi_2}\!\!\frac{1+t^2}{(x_2-t)(t-x_1)}\,\DI\nu(t)\Big]>\, \xi_2-\xi_1.
$$
It follows that necessarily $\beta\ge1$, with $\beta=1$ only if $\nu([\xi_1,\xi_2])=0$. Thus, $f'(\tau)\le 1$ and the equality $f'(\tau)=1$ holds only if $f\in\Aut(\UD)$. Recalling that, by the condition, $f$ fixes three points on~$\UC$ completes the proof of (i).

\parM{Proof of (ii)} Arguing in a similar way as in the proof of~(i), we see that $\Phi$ extends holomorphically through~$(\xi_1,\xi_2)$ to the lower half-plane  and that $0<\Phi(x_2)-\Phi(x_1)<\xi_2-\xi_1$ whenever $\xi_1<x_1<x_2<\xi_2$. The Nevanlinna Representation takes the following form
$$
\Phi(z)=\alpha\,+\,\beta z\,+\!\!\!\!\int\limits_{\Real\setminus(\xi_1,\xi_2)}\!\!\!\frac{1+tz}{t-z}\,\DI\nu(t)\quad\text{for all~$z\in \Complex\setminus\big(\Real\setminus(\xi_1,\xi_2)\big)$,}
$$
where $\alpha\in\Real$, $\beta:=1/f'(\tau)\ge0$ and $\nu$ is a finite positive Borel measure on $\Real\setminus(\xi_1,\xi_2)$.
Thus from
$$
\Phi(x_2)-\Phi(x_1)=(x_2-x_1)\Big[\beta\,+\!\!\!\! \int\limits_{\Real\setminus(\xi_1,\xi_2)}\!\!\!\frac{1+t^2}{(x_2-t)(x_1-t)}\,\DI\nu(t)\Big]<\, \xi_2-\xi_1
$$
for all $x_1\in(\xi_1,\xi_2)$ and all~$x_2\in(x_1,\xi_2)$ it follows that $\beta\le 1$ and that $\beta=1$ if and only if~$\nu(\Real\setminus(\xi_1,\xi_2))=0$, which finishes the proof.
\end{proof}

\begin{remark}\label{RM_continuity}
In the proof of Theorem~\ref{TH_main} we will deal with monotonic functions on~$[0,1]^n$.
Note that if a function $\mu:[0,1]^n\to\R$ is monotonic and continuous separately in each variable, then it is (jointly) continuous. Indeed, to fix the idea assume that $\mu$ is increasing in each variable. Using ``'variable-wise'' continuity, for any given $a\in(0,1)^n$ and $\varepsilon>0$ one can find $\delta\in(0,+\infty)^n$ such that $a\pm\delta\in[0,1]^n$ and $|\mu(a\pm\delta)-\mu(a)|<\varepsilon$. Then by the monotonicity, $\mu(a)-\varepsilon<\mu(a-\delta)\le\mu(b)\le\mu(a+\delta)<\mu(a)+\varepsilon$ whenever $|b_j-a_j|\le \delta_j$ for all~$j=1,\ldots,n$. A similar argument applies to boundary points of~$[0,1]^n$.
\end{remark}

\subsection{Proof of Theorem~\ref{TH_main}}\label{SS_proofs}
According to Theorem~\ref{TH_embedd}, it is sufficient to show that for any $\phi\in\mathfrak S\setminus\{\id_\UD\}$ there exists an evolution family~$(\psi_{s,t})$ such that $\psi_{0,1}=\phi$ and $\psi_{s,t}\in\mathfrak S$ whenever ${0\le s\le t\le 1}$.

If $\phi\in\Aut(\UD)$, then $\phi=\phi_1$ for some one-parameter semigroup~$(\phi_t)\subset\Aut(\UD)$ having the same BRFPs and the DW-point as~$\phi$ does. In such a case, the functions $\psi_{s,t}:=\phi_{t-s}$ form the desired evolution family.

Therefore, we suppose that $\phi\not\in\Aut(\UD)$. Let us first consider the case, which appears to be the simplest.

\caseU{(a)}{$\mathfrak S=\UF[F]$ with $n\le3$}
By Lemma~\ref{LM_embedd} there exists an evolution family $(\varphi_{s,t})$ such that $\varphi_{0,1}=\phi$.
Then by EF2, $\phi=\varphi_{u,1}\circ\varphi_{0,u}$ for all $u\in[0,1]$. Hence by Lemma~\ref{LM_ChainRule}, $\sigma_j$'s are regular contact points of~$\varphi_{0,u}$ for all~$u\in[0,1]$. Using~\cite[Theorem~3.5]{BRFPLoewTheory} we see that $[0,1]\ni u\mapsto \varphi_{0,u}(\sigma_j)$ is absolutely continuous for each~$j=1,\ldots,n$. Moreover, for all $u\in[0,1]$, the points $\varphi_{0,u}(\sigma_j)$ are pairwise distinct because $\varphi_{0,u}$ is univalent in~$\UD$ and conformal at each $\sigma_j$ (in the sense of~\cite[\S4.3]{Pommerenke2}). Recall that $n\le3$ by our assumption. Note also that $\varphi_{0,1}(\sigma_j)=\phi(\sigma_j)=\sigma_j$ for all~$j$'s.  Therefore, it is easy to construct a family $(g_u)_{u\in[0,1]}\subset\Aut(\UD)$ such that: (i)~$g_u(\varphi_{0,u}(\sigma_j))=\sigma_j$ for each~$j=1,\ldots,n$; (ii)~$g_1=\id_\UD$; (iii)~the functions ${[0,1]\ni u\mapsto g_u(0)}$ and ${[0,1]\ni u\mapsto g'_u(0)}$ are absolutely continuous. For all $t>1$ we set $g_t:=\id_\UD$. Finally, define $\psi_{s,t}:=g_t\circ\varphi_{s,t}\circ g_s^{-1}$ for all $s\ge0$ and all $t\ge s$. Then by~\cite[Lemma~2.8]{SMP}, $(\psi_{s,t})$ is an evolution family. Moreover, by construction $\psi_{0,1}=\phi$ and $\psi_{s,t}\in\mathfrak S$ whenever ${0\le s\le t\le 1}$. Thus for case~(a) the proof is complete.

\caseU{(b)}{$\mathfrak S=\Ut\tau[F]$ with $\tau\in\UC$ and $\Card(F)\le2$}
 If $n=1$, then the proof is essentially the same as in Case~(a). We only have to replace the conditions upon the family~$(g_u)$ by the following:
\begin{itemize}
\item[(i)] $g_u(\varphi_{0,u}(\sigma_1))=\sigma_1$ and $g_u(\varphi_{0,u}(\tau))=\tau$ for all $u\in[0,1]$;
\item[(ii)] $g'_u(\varphi_{0,u}(\tau))=\big(1+u(\phi'(\tau)-1)\big)/\varphi'_{0,u}(\tau)$ for all $u\in[0,1]$,
\end{itemize}
and notice that by~\cite[Theorem~3.5]{BRFPLoewTheory} also the functions $[0,1]\ni u\mapsto\varphi'_{0,u}(\sigma_1)$ and $[0,1]\ni u\mapsto\varphi'_{0,u}(\tau)$ are absolutely continuous.

So let us assume that $n=2$. Denote by~$L'_0$, $L'_1$ and $L'_2$ the pairwise disjoint open arcs of~$\UC$ contained between~$\sigma_1$ and $\sigma_2$, $\sigma_1$ and $\tau$, $\sigma_2$ and $\tau$, respectively.
Denote by~$J$ the set of indices~$j=0,1,2$ for which $\phi$ can be extended continuously to~$L'_j$ with $\phi(L'_j)\subset\UC$. Note that~$\phi'(\tau)\le1$ and $\phi\neq\id_\UD$. Hence, by Lemma~\ref{LM_3p_Loewner}(ii), $J\not\ni 0$.
First we assume that $J=\{1,2\}$. Then apply Case~(a) with~$\{\sigma_1,\sigma_2,\tau\}$ substituted for~$F$. The functions $\psi_{s,t}$, ${0\le s\le t\le 1}$, satisfy the hypothesis of Lemma~\ref{LM_3p_Loewner}(i). Thus $\psi_{s,t}'(\tau)\le 1$, which means that~$\psi_{s,t}\in\Ut\tau[\{\sigma_1,\sigma_2\}]$ whenever ${0\le s\le t\le 1}$.

Now we assume that $\{1,2\}\not\subset J$. Swapping $\sigma_1$ and~$\sigma_2$ if necessary, we have $L_1=L'_0$ and $L_2=L'_1\cup\{\tau\}\cup L'_2$.
Therefore, $\phi$ satisfies the hypothesis of Lemma~\ref{LM_multiparam_family}. Let $(f_a)_{a\in[0,1]^2}$ be the family constructed in this lemma with $z_0:=0$. Recall that by \cite[Theorem~4.14 on p.\,83]{Pommerenke2}, for each $a\in[0,1]^2$, $f_a$ has three pairwise distinct regular contact points $\xi_a^1$, $\xi_a^2$, $\xi_a$ such that $f(\xi_a^j)=\sigma_j$, $j=1,2$, and $f_a(\xi_a)=\tau$. Consider the unique $g_a\in\Aut(\UD)$ such that $g_a(\tau)=\xi_a$ and $g_a(\sigma_j)=\xi_a^j$, $j=1,2$. Then $h_a:=f_a\circ g_a\in\UF[\{\tau,\sigma_1,\sigma_2\}]$ for all~$a\in[0,1]^2$. Now for~$a,b\in[0,1]^2$ with $a_1\le b_1$ and $a_2\le b_2$, let $\psi_{a,b}:=h^{-1}_b\circ h_a$.

Now we are going to prove
\claimM{A}{If $[0,+\infty)\ni t\mapsto a_j(t)\in[0,1]$, $j=1,2$, are non-decreasing absolutely continuous functions with $a_j(0)=0$, $a_j(t)=1$ for all $t\ge 1$, then the functions $\psi_{a(s),a(t)}$, where $a(t):=\big(a_1(t),a_2(t)\big)$, form an evolution family, with $\psi_{a(0),a(1)}=\phi$.}
The family $\phi_{s,t}:=f_{a(t)}^{-1}\circ f_{a(s)}$, $t\ge s\ge 0$, is an evolution family because, by properties (i), (iii) and~(v) from Lemma~\ref{LM_multiparam_family}, $\phi_{s,t}$'s are holomorphic self-maps of~$\UD$ obviously satisfying conditions EF1 and~EF2 and such that $\phi_{s,t}(0)=0$, $\phi'_{s,t}(0)>0$, and $[0,+\infty)\ni t\mapsto |\phi'_{0,t}(0)|$ is absolutely continuous. Moreover, for each~$t\ge0$, $f_{a(t)}\circ\phi_{0,t}=f_{a(0)}=f_{(0,0)}=\phi$ and hence with help of Lemma~\ref{LM_ChainRule} we see that $\phi_{0,t}$ has regular contact points at $\tau$, $\sigma_1$ and $\sigma_2$ and that the angular limits of $\phi_{0,t}$ at these points are equal to $f_{a(t)}^{-1}(\tau)=g_{a(t)}(\tau)$, $f_{a(t)}^{-1}(\sigma_1)=g_{a(t)}(\sigma_1)$, and $f_{a(t)}^{-1}(\sigma_2)=g_{a(t)}(\sigma_2)$, respectively. Finally, $\psi_{a(s),a(t)}=g_{a(t)}^{-1}\circ \phi_{s,t}\circ g_{a(s)}$ for all $s\ge0$ and~$t\ge s$. Thus, as above, we can use \cite[Theorem~3.5]{BRFPLoewTheory} and \cite[Lemma~2.8]{SMP} to make sure that~$(\psi_{a(s),a(t)})$ is also an evolution family. Finally, by construction, $h_{(1,1)}=\id_\UD$ and $h_{(0,0)}=f_{(0,0)}=\phi$. Hence $\psi_{a(0),a(1)}=h_{(1,1)}^{-1}\circ h_{(0,0)}=\phi$. This proves Claim A.\vskip1ex

Note that $\psi_{a,b}$ satisfies the hypothesis of Lemma~\ref{LM_3p_Loewner}(i) when~$a_2=b_2$ and the hypothesis of Lemma~\ref{LM_3p_Loewner}(ii) when~$a_1=b_1$. For any ${a\in[0,1]^2}$ denote $\lambda(a):=\psi'_{(0,0),a}(\tau)$.  Fix an arbitrary~$a_2^0\in[0,1]$. With the help of Remark~\ref{RM_Chain-Rule} and Lemma~\ref{LM_3p_Loewner} we conclude that $\lambda(\cdot,a_2^0)$ is strictly decreasing on~$[0,1]$. Consider now the map $a:[0,+\infty)\to[0,1]^2$ given by ${a(t):=(0,2t)}$ if ${0\le t\le \frac12a_2^0}$, ${a(t):=(2t-a_2^0,a_2^0)}$ if ${\frac12a_2^0\le t\le\frac12(1+a_2^0)}$, ${a(t):=(1,2t)}$ if ${\frac12(1+a_2^0)\le t\le1}$, and ${a(t):=(1,1)}$ otherwise. By construction, $(\psi_{a,b})\subset\mathfrak P[\{\tau,\sigma_1,\sigma_2\}]$ and hence from Claim~A and~\cite[Theorem~1.1]{BRFPLoewTheory} it follows that $\lambda(\cdot,a_2^0)$ is absolutely continuous on~$[0,1]$.
Similar argument applies to $\lambda(a_1^0,\cdot)$ for any~$a_1^0\in[0,1]$. Thus we have proved the following
\claimM{B}{The function $\lambda(a_1,a_2):=\psi'_{(0,0),(a_1,a_2)}(\tau)$, $a_1,a_2\in[0,1]$, is strictly decreasing in~$a_1$ and strictly increasing in~$a_2$. Moreover, it is absolutely continuous in each variable.}

Using Claim~$B$ and taking into account Remark~\ref{RM_continuity}, it is easy to show that there exists $a_1^0\in[0,1)$ and a continuous strictly increasing function $a_2^*:[a_1^0,1]\to[0,1]$ with $a_2^*(a_1^0)=0$ and $a_2^*(1)=1$ such that $\lambda(a_1,a_2^*(a_1))=\phi'(\tau)$ for any~$a_1\in[a_1^0,1]$. We extend it to $[0,a^*_1(0)]$ by setting $a_2^*|_{[0,a^0_1]}\equiv0$.

The function $A(a_1):=[a_1+a_2^*(a_1)]/2$ is strictly increasing and continuous on~$[0,1]$, with $A(0)=0$ and $A(1)=1$. Therefore, it has the continuous and strictly increasing inverse $[0,1]\ni t\mapsto a_1(t)\in[0,1]$. The map ${a^*:[0,1]\to[0,1]^2;}~t\mapsto \big(a_1(t),a_2^*(a_1(t))\big)$ is Lipschitz continuous. Indeed, if $0\le t_1\le t_2\le1$, then $\delta_1:=a_1(t_2)-a_1(t_1)\ge0$, $\delta_2:=a_2^*(a_1(t_2))-a_2^*(a_1(t_1))\ge0$ and $\delta_1+\delta_2=2(t_2-t_1)$.

Now apply Claim~A to $(a_1(t),a_2(t)):=a^*(t)$, where as before we extend $a^*$ to $(1,+\infty)$ by $a^*|_{(1,+\infty)}\equiv(1,1)$. Recall that $(\psi_{a,b})\subset\mathfrak P[\{\tau,\sigma_1,\sigma_2\}]$. Therefore, to complete the proof for Case~(b)  it remains  to notice that by construction and by Claim~B, $t\mapsto\lambda(a^*(t))$ is non-increasing, so that $\psi_{a^*(s),a^*(t)}'(\tau)\le 1$ whenever~$t\ge s\ge0$.

\caseU{(c)}{$\mathfrak S=\Ut\tau[F]$ with $\tau\in\UD$}
If $n=1$, then again the proof is the same as in Case~(a), except that the conditions on the family~$(g_u)$ are replaced by $g_u(\varphi_{0,u}(\sigma_1))=\sigma_1$ and $g_u(\varphi_{0,u}(\tau))=\tau$ for all $u\in[0,1]$ and that we additionally notice that, according to condition EF3 in Definition~\ref{DF_evol_family}, the map $[0,1]\ni u\mapsto \varphi_{0,u}(\tau)\in\UD$ is absolutely continuous.

Therefore, we may assume that $n\ge 2$. Replacing $\phi$ by $g^{-1}\circ\phi\circ g$ with a suitable ${g\in\Aut(\UD)}$, we may also suppose that~$\tau=0$. Then by Loewner's Lemma (Theorem~\ref{TH_LM-Loewner}), the function~$\phi$ satisfies the hypothesis of Lemma~\ref{LM_multiparam_family}. Main step of the proof in Case~(c) is the following statement concerning the family $(f_a)$ constructed in Lemma~\ref{LM_multiparam_family}.

\claimM{C}{There exists a Lipschitz continuous map $a^*:[0,1]\to[0,1]^n$ with non-de\-creasing components $a^*_j$, $j=1,\ldots,n,$ such that  $a^*(0)={(0,\ldots,0)}$, $a^*(1)={(1,\ldots,1)}$, and  $f_{a^*(t)}\circ(z\mapsto e^{i\theta(t)}z)\in\mathfrak P_0[F]$ for all~$t\in[0,1]$ and some $\theta:[0,1]\to\Real$.}
Using Claim~C and arguing essentially in the same way as in the proof of Claim~A, it is not difficult to show that the functions $\psi_{s,t}(z):=e^{-i\theta(t)}(f_{a^*(t)}^{-1}\circ f_{a^*(s)})(e^{i\theta(s)}z)$, $t\ge s\ge0$, where we extend $a^*$ and $\theta$ by setting $a^*|_{[1,+\infty)}\equiv(1,\ldots,1)$,  $\theta|_{[1,+\infty)}\equiv\theta(1)$, form the desired evolution family: $\psi_{0,1}=\phi$ and $\psi_{s,t}\in\mathfrak P_0[F]$ for all $s,t\in[0,1]$ with $s\le t$.

It remains to prove Claim~C. For $a\in[0,1]^n$ and $j\in[1,n]\cap\N$, denote by $L_j(a)$ the arc of~$\UC$ going in the counter-clockwise direction from~$f^{-1}(\sigma_j)$ to $f^{-1}(\sigma_{j+1})$. The condition  $f_{a^*(t)}\circ(z\mapsto e^{i\theta(t)}z)\in\mathfrak P_0[F]$ from Claim~C can be rephrased as $|L_j(a^*(t))|=|L_j|$ for all~$t\in[0,1]$ and all $j=1,\ldots,n$. Denote $\ell_j(a):=|L_j(a)|$ and $\varphi_{a,b}:=f_b^{-1}\circ f_a$ for all $a,b\in[0,1]^n$ with $a_j\le b_j$ for any~$j=1,\ldots,n$. Applying Loewner's Lemma (Theorem~\ref{TH_LM-Loewner}) to the functions $\varphi_{a,a+\delta e_j}$, where $\delta>0$ and $\{e_j\}_{j=1}^n$ is the standard basis in~$\R^n$, and taking into account that $\sum_{j=1}^n\ell_j\equiv2\pi$, we see that $\ell_j(a_1,\ldots,a_n)$ is strictly decreasing in~$a_j$ and strictly increasing in $a_k$ if~${k\neq j}$.
Moreover, again by essentially the same argument as in the proof of Claim~A, $\ell_j$'s are continuous in each variable.
This allows us construct~$a^*$ as follows.

Recall that $f_{(1,\ldots,1)}=\phi$ and hence $\ell_j(1)=\ell_j(0)=|L_j|$ for all~$j=1,\ldots,n$. Therefore, thanks to the continuity and monotonicity of~$\ell_j$, there exists a map ${\tilde a_1:[0,1]^{n-1}\to[0,1]}$ continuous and strictly increasing in each variable such that we have ${\tilde a_1(0,\ldots,0)}=0$, ${\tilde a_1(1,\ldots,1)}=1$, and  $$\hphantom{\ell_1(\tilde a_1(\tilde a_2,a''),\tilde a_2(a''),a'')}\mathllap{\ell_1(\tilde a_1(a'),a')}=|L_1|\quad\text{ for all~$a':=(a_2,\ldots,a_n)\in[0,1]^{n-1}$.}$$

Fix $j\in[2,n]\cap\N$. Recall that $$\ell_j=2\pi-\ell_1-\sum\limits_{k\neq1,j}\ell_k$$ and that $\ell_k(a)$  is strictly increasing in~$a_j$ and in~$a_1$ for any $k\neq1,j$. Therefore, $a':=(a_2,\ldots,a_n)\mapsto\ell_j(\tilde a_1(a'),a')$ is decreasing in~$a_j$ and increasing in~$a_k$ if $k\neq j$. Repeating the above argument for $\ell_j$'s, $j=1,\ldots,n$, replaced by $a'\mapsto \ell_j(\tilde a(a'),a')$'s, $j=2,\ldots,n$, we conclude that there exists a map ${\tilde a_2:[0,1]^{n-2}\to[0,1]}$ continuous and strictly increasing in each variable such that ${\tilde a_2(0,\ldots,0)}=0$, ${\tilde a_2(1,\ldots,1)}=1$, and  $$\ell_j(\tilde a_1(\tilde a_2,a''),\tilde a_2(a''),a'')=|L_j|\quad\text{ for all~$a'':=(a_3,\ldots,a_n)\in[0,1]^{n-2}$ and~$j=1,2$.}$$

Repeat this procedure until we end up with a continuous map ${\hat a:[0,1]\to[0,1]^{n-1}}$ with strictly increasing components such that $\hat a(0)=(0,\ldots,0)$, $\hat a(1)=(1,\ldots,1)$, and $\ell_j(\hat a(a_n),a_n)=|L_j|$ for all $j:=1,\ldots,n$ and all $a_n\in[0,1]$.

It remains to set $a^*:=(\hat a\circ q^{-1},q^{-1})$, where $q^{-1}$ is the inverse of the function $$q(a_n):=\frac1n\Big(a_n+\sum_{j=1}^{n-1}\hat a_j(a_n)\Big),\quad a_n\in[0,1].$$
\proofbox

\subsection{Proof of Corollary~\ref{CR_specific}}
Theorems~1 and~2 imply that the semigroup $\mathfrak S:=\Ut{\,1}[\{-1\}]$ is the union of all evolution families $(\varphi_{s,t})$ lying in~$\mathfrak S$. Note that if $(\varphi_{s,t})$ is such an evolution family, then so is $(\tilde\varphi_{s,t}):=(\varphi_{s+s_0,t+s_0})$ for any~$s_0\ge0$. Therefore,
\begin{equation}\label{EQ_s=0}
\mathfrak S=\big\{\tilde\varphi_{0,t}:t\ge0,~(\tilde\varphi_{s,t})\subset\mathfrak S\text{\small{} is an evolution family}\}.
\end{equation}
For a given evolution family $(\tilde\varphi_{s,t})\subset\mathfrak S$, consider the function of $\Lambda(t):=\log\tilde\varphi_{0,t}'(\sigma)$, ${\sigma:=-1}$. By the very construction, $\Lambda$ is non-decreasing. Moreover, if $\Lambda(t)=\Lambda(s)$ for some ${t>s\ge0}$, then $\tilde\varphi_{s,t}'(\sigma)=1$, see Remark~\ref{RM_Chain-Rule}, and hence $$\tilde\varphi_{0,t}=\tilde\varphi_{s,t}\circ\tilde\varphi_{0,s}=\tilde\varphi_{0,s},$$ because if we had $\mathfrak S\ni\tilde \varphi_{s,t}\neq\id_\UD$, then there would be two different DW-points of~$\tilde\varphi_{s,t}$, ${\tau=1}$ and ${\sigma=-1}$.
Therefore, for any $t_0>0$ there exists a family~$(\hat\varphi_{s,t})_{0\le s\le t\le T_0}$, $T_0:=\Lambda(t_0)$, such that $\hat\varphi_{\Lambda(s),\Lambda(t)}=\tilde\varphi_{s,t}$ whenever $0\le s\le t\le t_0$. Clearly, using, \textit{e.g.}, the one-parameter semigroup $(\phi_t)\subset\mathfrak S\cap\Aut(\UD)$, $\phi_t(z):=\tfrac{z+x_t}{1+x_tz}$, $x_t:=\tfrac{e^t-1}{e^t+1}$, we can extend the family $(\hat\varphi_{s,t})$ to all $t\ge s\ge0$ in such a way that it satisfies conditions EF1 and EF2 in Definition~\ref{DF_evol_family} and $\hat\varphi'_{0,t}(\sigma)=e^t$ for all $t\ge0$. Then by Theorem~\ref{TH_differentiability}, $(\hat\varphi_{s,t})$ is an evolution family.

Taking into account the above argument and making use of \cite[Theorem\,1.1]{BRFPLoewTheory}\,\footnote{Observe that our notation~$\Lambda$ differs coincides with the spectral function in~\cite{BRFPLoewTheory} taken with the \textsl{opposite sign}.}, \cite[Theorem\,1]{Goryainov-Kudryavtseva}, and \cite[Lemma\,1]{GoryainovBRFP}, we conclude that for any $T>0$, the set $${\mathfrak S_T:=\{\varphi\in\mathfrak S:}\,\,{\varphi'(-1)=e^{T}\}}$$ coincides with the set of all functions representable in the form $\varphi(z)=w_z(T)$ for all $z\in\UD$, where $w_z(t)$ is the unique solution to the initial value problem
\begin{equation*}
\ddt{w_z}=G_t(w_z):=\tfrac14(1-w_z)^2(1+w_z)q(w_z,t),\quad t\in[0,T],\qquad w_z(0)=z,
\end{equation*}
with some function $q:\UD\times[0,T]\to\C$ satisfying the following conditions:
\begin{itemize}
\item[(i)] for every $z\in\UD$, $q(z,\cdot)$ is measurable on~$[0,T]$;
\item[(ii)] for a.e.~$t\in[0,T]$, $q(\cdot,t)$ has the following integral representation
\begin{equation*}
q(z,t)=\int\limits_{\UC\setminus\{1\}}\!\!\frac{1-\kappa}{1+\kappa z}\,\DI\nu_t(\kappa),
\end{equation*}
\end{itemize}
where $\nu_t$ is a probability measure on $\UC\setminus\{1\}$, related to $\mu$ and $\alpha$ in~\cite[Theorem\,1]{Goryainov-Kudryavtseva} by $\nu_t=4\alpha\mu|_{\UC\setminus\{1\}}$, with $\alpha\mu(\UC\setminus\{1\})=\tfrac14$ because $G_t'(-1)=1$ for a.e.\,$t\in[0,T]$.

This is what was to be proved. \proofbox

\section{An open problem}\label{S_OpenProblem}
On the one hand, all elements of any evolution family are univalent in~$\UD$, see Remark~\ref{RM_univalent}. Hence a semigroup~$\mathfrak S\subset\Hol(\UD,\UD)$ can admit Loewner-type parametric representation only if actually $\mathfrak S\subset\mathfrak U$.

On the other hand,
 Theorems~\ref{TH_main} and~\ref{TH_embedd} suggest to conjecture that univalence of all elements in $\mathfrak S$ is ``essentially'' sufficient\footnote{Of course, in general case, some additional condition of topological nature would be also required.} for~$\mathfrak S$ to admit Loewner-type parametric representation. Recall also that for $n>3$, Theorem~\ref{TH_main} provides a Loewner-type parametric representation only for the semigroup~$\Ut\tau[F]$ with $\tau\in\UD$.

In this connection, the following question seems to be of considerable interest.

\vskip1ex
\noindent\textbf{Open problem.} \textit{Do the semigroups $\UF[F]$ and $\Ut\tau[F]$ admit Loewner-type parametric representation for any finite set~$F\subset\partial\UD$ and any $\tau\in\partial\UD\setminus F$?}
\vskip1ex

 \noindent Note that, in view of Theorems~\ref{TH_main} and~\ref{TH_embedd}, the affirmative answer for $\Ut\tau[F]$, $\tau\in\partial\UD$, would imply also that for $\UF[F]$.

\section*{Acknowledgement}
Much inspiration for the work under the present paper has been drawn from the scientific publications of and from the personal communication with Prof. Viktor V. Goryainov. The author is also grateful to Prof. Oliver Roth for pointing out reference~\cite{Renggli} essentially used in the proof. The text of the paper has been considerably improved thanks to valuable suggestions of the anonymous referees.


\begin{thebibliography}{99}
%
\bibitem{Abate}M.\,Abate, {\it Iteration theory of holomorphic maps on taut manifolds},
Research and Lecture Notes in Mathematics. Complex Analysis and Geometry,
Mediterranean, Rende, 1989.


\bibitem{Aleksandrov} I.\,A.\,Aleksandrov,
\textit{Parametric continuations in the theory of univalent functions}
(Russian), Izdat. ``Nauka'', Moscow, 1976. \hbox{MR0480952 (58 \#1099)}

\bibitem{Aleks1983}I.A. Aleksandrov, S.T. Aleksandrov\ and\ V.V. Sobolev, {\it Extremal properties
of mappings of a half plane into itself}, in {\it Complex analysis (Warsaw, 1979)}, 7--32, PWN,
Warsaw.

\bibitem{AleksST} S.T. Aleksandrov, {\it Parametric representation of functions univalent in the half plane}, in {\it Extremal problems of the theory of functions}, 3--10, Tomsk. Gos. Univ., Tomsk, 1979.

\bibitem{MilnVas}J. M. Anderson\ and\ A. Vasil'ev, {\it Lower Schwarz-Pick estimates and angular derivatives}, Ann. Acad. Sci. Fenn. Math. {\bf 33} (2008), no.~1, 101--110. MR2386840 (2009b:30014)

\bibitem{Berkson-Porta}E. Berkson and H. Porta, \textit{Semigroups of
holomorphic functions and composition operators,} Michigan Math. J. \textbf{25} (1978), 101--115.

\bibitem{Bhatia} R. Bhatia, {\it Matrix analysis}, Graduate Texts in Mathematics, 169, Springer, New York, 1997. MR1477662 (98i:15003)

\bibitem{Bauer} R.O. Bauer, \textit{Chordal Loewner families and univalent Cauchy transforms}, J. Math. Anal. Appl. \textbf{302} (2005), 484-501.

\bibitem{AlexanderClarkMeasures}F. Bracci, M. D. Contreras\ and\ S. D\'\i az-Madrigal, {\it Aleksandrov-Clark measures and semigroups of analytic functions in the unit disc}, Ann. Acad. Sci. Fenn. Math. {\bf 33} (2008), no.~1, 231--240. MR2386848 (2009c:30097)

\bibitem{BCM1} F. Bracci, M. Contreras, S. D\'iaz-Madrigal, {\it Evolution Families and the Loewner Equation I: the unit disc}. J. Reine Angew. Math. (Crelle's Journal), {\bf 672} (2012), 1--37.


\bibitem{BRFPLoewTheory} F.~Bracci, M.\,D. Contreras, S.~D\'\i az-Madrigal, and P.~Gumenyuk,
 {\it Boundary regular fixed points in Loewner theory}, Ann. Mat. Pura Appl. (4) {\bf 194} (2015), no.~1, 221--245.

\bibitem{Fractional} F.~Bracci and P.~Gumenyuk, \textit{Contact points and fractional singularities for semigroups of holomorphic self-maps in the unit disc,} 29 pp. To appear in J. Anal. Math. (2016); available at ArXiv:1309.2813 [math.CV]

\bibitem{deBranges} L. de Branges, A proof of the Bieberbach conjecture,
Acta Math. {\bf 154} (1985), no.~1-2, 137--152. MR0772434 (86h:30026)

\bibitem{Chon} I. Chon, {\it Infinitesimally generated stochastic totally positive matrices}, Commun. Korean Math. Soc. {\bf 12} (1997), no.~2, 269--273. MR1641869

\bibitem{SMP} M.\,D. Contreras, S.~D\'\i az-Madrigal, and P.~Gumenyuk, \textit{Loewner chains in the unit disk}. Rev. Mat. Iberoam. \textbf{26} (2010), 975--1012.

\bibitem{CDP1}M. D. Contreras, S. D\'\i az-Madrigal\ and\ C. Pommerenke, {\it Fixed points and boundary behaviour of the Koenigs function}, Ann. Acad. Sci. Fenn. Math. {\bf 29} (2004), no.~2, 471--488. MR2097244 (2006b:30059)

\bibitem{CDP2} M. D. Contreras, S. D\'\i az Madrigal\ and\ Ch. Pommerenke, {\it On boundary critical points for semigroups of analytic functions}, Math. Scand. {\bf 98} (2006), no.~1, 125--142. MR2221548 (2007b:30026)


\bibitem{CowenPommerenke} C. C. Cowen\ and\ C. Pommerenke, {\it Inequalities for the angular derivative of an analytic function in the unit disk}, J. London Math. Soc. (2) {\bf 26} (1982), no.~2, 271--289. MR0675170 (84a:30006)


\bibitem{ElinShoikhetTarkhanov}M. Elin, D. Shoikhet\ and\ N. Tarkhanov, {\it Separation of boundary singularities for holomorphic generators}, Ann. Mat. Pura Appl. (4) {\bf 190} (2011), no.~4, 595--618. MR2861061

\bibitem{Frolova} A. Frolova, M. Levenshtein, D. Shoikhet, A. Vasil'ev, {\it Boundary distortion estimates for holomorphic maps}, Complex Anal. Oper. Theory {\bf 8} (2014), no.~5, 1129--1149. MR3208806


\bibitem{GoryainovFr}V. V. Gorya\u\i nov, {\it Fractional iterates of functions that are analytic in the unit disk with given fixed points} Mat. Sb. {\bf 182} (1991), no. 9, 1281--1299; translation in Math. USSR-Sb. {\bf 74} (1993), no.~1, 29--46. MR1133569 (92m:30049)

\bibitem{Goryainov}V.V. Goryainov, \textit{Semigroups of conformal mappings,}
Mat. Sb. (N.S.) {\bf 129(171)} (1986), no.~4, 451--472 (Russian); translation
in Math. USSR Sbornik \textbf{57} (1987), 463--483.

\bibitem{Goryainov1996}
V.\,V. Goryainov, \textit{Evolution families of analytic functions and time-inhomogeneous
Markov branching processes}, Dokl. Akad. Nauk \textbf{347}(1996), No.\,6,
729--731; translation in Dokl. Math. \textbf{53}(1996), No.\,2, 256--258.

\bibitem{GoryainovDiff} V.\,V.~Goryainov, \textit{Evolution families of conformal mappings with fixed points.} (Russian. English summary). Z\'\i{}rnik Prats' Instytutu Matematyky NAN Ukrayiny. National Academy of Sciences of Ukraine (ISSN 1815-2910) {\bf 10} (2013), No.4-5, 424-431. Zbl~1289.30024

\bibitem{GoryainovObzor}V.\,V. Goryainov, \textit{Semigroups of analytic functions in analysis and applications}, Uspekhi Mat. Nauk {\bf 67} (2012), no. 6(408), 5--52; translation in Russian Math. Surveys {\bf 67} (2012), no.~6, 975--1021

\bibitem{GoryainovBRFP} V.\,V. Goryainov, \textit{Evolution families of conformal mappings with fixed points and the L\"owner-Kufarev equation}, Mat. Sb. {\bf 206} (2015), no. 1, 39--68; translation in Sb. Math. {\bf 206} (2015), no.~1-2, 33-60.

\bibitem{Goryainov-Ba}V.\,V. Goryainov and I. Ba, \textit{Semigroups of
conformal mappings of the upper half-plane into itself with hydrodynamic
normalization at infinity,} Ukrainian Math. J. \textbf{44} (1992), 1209--1217.

\bibitem{Goryainov-Kudryavtseva} V.\,V.~Goryainov, O.\,S.~Kudryavtseva, {\it One-parameter semigroups of analytic functions, fixed points and the Koenigs function}, Mat. Sb. \textbf{202} (2011), No.\,7, 43--74 (Russian); translation in Sbornik: Mathematics,
\textbf{202} (2011), No.\,7-8, 971--1000.

\bibitem{Gutljanski} V. Ja. Gutljanski\u\i, Parametric representation of univalent functions (in Russian),
Dokl. Akad. Nauk SSSR {\bf 194} (1970), 750--753. MR0271324 (42 \#6207); English translation in
Soviet Math. Dokl. 11 (1970), 1273--1276



\bibitem{LieSemigroups} K. H. Hofmann\ and\ J. D. Lawson, {\it Local semigroups in Lie groups and locally reachable sets}, Rocky Mountain J. Math. {\bf 20} (1990), no.~3, 717--735. MR1073719 (91i:93013)


\bibitem{Kufarev1943}
P.\,P. Kufarev, {\it On one-parameter families of analytic
functions} (in Russian. English summary), Rec.~Math. [Mat.
Sbornik] N.S. \textbf{13 (55)} (1943), 87--118.









\bibitem{Loewner}K. L\"{o}wner, Untersuchungen \"{u}ber schlichte
konforme Abbildungen des Einheitskreises, Math. Ann. \textbf{89} (1923),
103--121.

\bibitem{LoewnerMatrices} C. Loewner, \textit{On totally positive matrices,} Math. Z. {\bf 63} (1955), 338--340. MR0073657 (17,466f)

\bibitem{LoewnerSeminar} C. Loewner, Seminars on analytic functions, Institute for Advanced Study, Princeton, New Jersey, vol. 1, (1957). Available at \verb+http://babel.hathitrust.org+

\bibitem{LoewnerMonotone} C. Loewner, On generation of monotonic transformations of higher order by infinitesimal transformations, J. Analyse Math. {\bf 11} (1963), 189--206. MR0214711 (35 \#5560)

\bibitem{Poggi} P. Poggi-Corradini, {\it Canonical conjugations at fixed points other than the Denjoy-Wolff point}, Ann. Acad. Sci. Fenn. Math. {\bf 25} (2000), no.~2, 487--499. MR1762433 (2001f:30033)

\bibitem{Pommerenke-65}Ch. Pommerenke, \textit{\"{U}ber dis subordination
analytischer funktionen}, J. Reine Angew Math. \textbf{218} (1965), 159--173.

\bibitem{Pommerenke}Ch.\,Pommerenke, {\it Univalent functions}.
With a chapter on quadratic differentials by Gerd Jensen, Vandenhoeck \&
Ruprecht, G\"{o}ttingen, 1975.

\bibitem{Pommerenke2}Ch. Pommerenke, {\it Boundary behaviour of conformal mappings}. Springer-Verlag, 1992.

\bibitem{PommVas} C. Pommerenke\ and\ A. Vasil'ev, {\it Angular derivatives of bounded univalent functions and extremal partitions of the unit disk}, Pacific J. Math. {\bf 206} (2002), no.~2, 425--450. MR1926785 (2003i:30024)

\bibitem{Renggli} H. Renggli, {\it An inequality for logarithmic capacities}, Pacific J. Math. {\bf 11} (1961), 313--314. MR0155992


\bibitem{Schramm}O. Schramm, Scaling limits of loop-erased random
walks and uniform spanning trees, Israel J. Math. \textbf{118} (2000),
221--288.


\bibitem{Unkelbach1938} H. Unkelbach, {\it\"Uber die Randverzerrung bei konformer Abbildung}, Math. Z. {\bf 43} (1938), 739--742. Zbl:~0018.22403

\bibitem{Unkelbach1940} H. Unkelbach, {\it\"Uber die Randverzerrung bei schlichter konformer Abbildung}, Math. Z. {\bf 46} (1940), 329--336. MR0002607 (2,83d)

\end{thebibliography}
\end{document}